\documentclass{amsart}
\usepackage[utf8]{inputenc}
\usepackage[margin=1 in]{geometry}
\usepackage{colonequals}
\usepackage{amssymb,bm,amsfonts, stmaryrd, amsmath, amsthm, mathtools,comment}
\usepackage{enumitem}
\usepackage{multicol}
\usepackage[all]{xy}
\makeatletter
\@namedef{subjclassname@2020}{\textup{2020} Mathematics Subject Classification}
\makeatother
\usepackage[colorlinks]{hyperref}
\hypersetup{
 colorlinks=true,
 linkcolor=blue,
 filecolor=magenta, 
 urlcolor=cyan,
}
\usepackage{float}
\usepackage{tabularray}
\usepackage{mathrsfs}
\usepackage[capitalise]{cleveref}

\crefformat{equation}{(#2#1#3)}
\crefrangeformat{equation}{(#3#1#4--#5#2#6)}
\crefformat{enumi}{(#2#1#3)}
\crefrangeformat{enumi}{(#3#1#4--#5#2#6)}

\usepackage{wrapfig, tikz, tikz-cd}
\usepackage{caption}
\usepackage{subcaption}
\usetikzlibrary{arrows}

\newtheorem{theorem}{Theorem}[section]
\newtheorem{lemma}[theorem]{Lemma}
\newtheorem{proposition}[theorem]{Proposition}
\newtheorem{corollary}[theorem]{Corollary}
\newtheorem{conjecture}[theorem]{Conjecture}
\newcounter{intro}
\newtheorem{questionx}{Question}
\newtheorem{introthm}[intro]{Theorem}

\theoremstyle{definition}
\newtheorem{definition}[theorem]{Definition}
\newtheorem{example}[theorem]{Example}
\newtheorem{remark}[theorem]{Remark}

\newtheorem{construction}[theorem]{Construction}
\newtheorem{notation}[theorem]{Notation}

\newtheorem{chunk}[theorem]{}

\newtheorem{thm}{Theorem}[subsection]

\theoremstyle{definition}
\newtheorem{ex}[thm]{Example}
\newtheorem{rem}[thm]{Remark}

\newtheorem{ch}[thm]{}

\definecolor{bettergreen}{RGB}{50, 170, 50}

\definecolor{betterblue}{RGB}{100, 149, 236}

\DeclareMathOperator{\lcm}{lcm}

\newcommand{\ges}{\geqslant}
\newcommand{\les}{\leqslant}

\newcommand{\spec}{{\operatorname{Spec}^*}}

\newcommand{\Hom}{{\operatorname{Hom}}}

\newcommand{\Ext}{{\operatorname{Ext}}}
\newcommand{\Tor}{{\operatorname{Tor}}}

\newcommand{\del}{\partial}
\renewcommand{\H}{\operatorname{H}}
\renewcommand{\S}{\mathcal{S}}

\DeclareMathOperator{\codepth}{codepth}
\DeclareMathOperator{\coker}{coker}
\DeclareMathOperator{\depth}{depth}
\DeclareMathOperator{\ad}{ad}
\newcommand{\supp}{{\operatorname{Supp}}}
\newcommand{\xra}{\xrightarrow}

\newcommand{\f}{\bm{f}}
\newcommand{\V}{{\rm{V}}}
\newcommand{\sym}{{\rm{Sym}}}

\newcommand{\cV}{{\mathcal{V}}}
\newcommand{\cC}{{\mathcal{C}}}
\renewcommand{\ll}{\ell\ell}

\newcommand{\ann}{{\operatorname{ann}}}
\newcommand{\pdim}{{\mathrm{projdim}}}
\newcommand{\shift}{{\mathsf{\Sigma}}}

\DeclareMathOperator{\height}{height}
\DeclareMathOperator{\cid}{cid}
\DeclareMathOperator{\rank}{rank}
\DeclareMathOperator{\rad}{rad}

\newcommand{\m}{\mathfrak{m}}
\newcommand{\p}{\mathfrak{p}}
\newcommand{\A}{\mathbb{A}}

\DeclareMathOperator{\sgn}{sign}

\newcommand*\circled[1]{\tikz[baseline=(char.base)]{ \node[shape=circle,draw,inner sep=2pt] (char) {#1};}}
\newcommand*\squared[1]{\tikz[baseline=(char.base)]{ \node[shape=rectangle,draw,inner sep=2pt] (char) {#1};}}

\DeclareMathOperator{\im}{im}
\title[The deformation problem for monomial ideals]{The embedded deformation problem for monomial ideals}
\author[B.~Briggs]{Benjamin Briggs}
\address{Department of Mathematics, Imperial College London, South Kensington Campus, London, SW7
2AZ, UK}
\email{b.briggs@imperial.ac.uk}
\author[E.~Grifo]{Elo\'{i}sa Grifo}
\address{Department of Mathematics,
University of Nebraska, Lincoln, NE 68588, U.S.A.}
\email{grifo@unl.edu}
\author[J.~Pollitz]{Josh Pollitz}
\address{Mathematics Department, 
Syracuse University, 
Syracuse, NY 13244 U.S.A.}
\email{jhpollit@syr.edu}


\begin{document}

\begin{abstract}
This article is concerned with homological properties of local or graded rings whose defining relations are monomials on some regular sequence. 
The main result of the article positively answers a question of Avramov for such a ring $R$. More precisely, we establish that an embedded deformation of $R$ corresponds exactly to a degree two central element in the homotopy Lie algebra of $R$, as well as a free summand of the conormal module of $R$. A major input in the proof is an analysis of cohomological support varieties. Other main results include establishing a lower bound for the dimension of the cohomological support variety of any complex over such rings, and classifying all possible subvarieties of affine $n$-space that are the cohomological support of rings defined by $n$ monomial relations where $n$ is five or less.
\end{abstract}

\maketitle

\vspace{-.2in}
 
\section*{Introduction}\label{s_intro}

In the late 1980s, Avramov~\cite{Avramov:1989a} asked whether an embedded deformation of a local ring corresponds to a degree $2$ central element in the homotopy Lie algebra. To make this precise, let $R$ be a commutative noetherian local ring with residue field $k$. Recall that $R$ admits an embedded deformation if $\widehat{R}=S/(f)$ for some local ring $S$ with $f$ a regular element in the square of the maximal ideal of $S$. More generally, $R$ admits an embedded deformation of codimension $c$ if $\widehat{R}=S/(\f)$ for some local ring $S$ with $\f=f_1,\ldots,f_c$ a regular sequence in the square of the maximal ideal of $S$. 

The \emph{homotopy Lie algebra} of $R$, denoted $\pi(R)$, is a positively graded Lie algebra over $k$ that has seen a number of striking applications~\cite{Avramov1996,Avramov:1999,Avramov/Halperin:1987,Briggs:2022}. Its defining property makes its importance evident: the graded algebra $\Ext_R(k,k)$ is the universal envelope of $\pi(R)$. Moreover, the maximal codimension of an embedded deformation of $R$ is bounded above by the $k$-vector space rank of $\pi^2(R)$. Avramov's question~is: \looseness -1
\begin{questionx}\label{Lucho_question} 
Are conditions \ref{emb_def} and \ref{central_elt} below equivalent?
    \begin{enumerate}[label=(\roman*)]
        \item\label{emb_def} The ring $R$ admits an embedded deformation of codimension $c$.
        \item \label{central_elt} The rank of the subspace of $\pi^2(R)$ consisting of central elements of $\pi(R)$ is $c$.
    \end{enumerate}
\end{questionx}

Recall that an element $z$ in a Lie algebra $L$ is central if $\ad(z)=[z,-]$ is identically zero on $L$. One can directly prove that \ref{emb_def} always implies \ref{central_elt}. Furthermore, part of the motivation for the question is that the extremal case is known to hold by a theorem of Avramov--Halperin~\cite[Theorem~C]{Avramov/Halperin:1987}: \emph{$R$ admits a maximal embedded deformation  (i.e.\ $R$  is complete intersection) if and only if each element of $\pi^2(R)$ is central in $\pi(R)$.} As further evidence, the answer to \cref{Lucho_question} is ``yes" in the following cases: 
\begin{itemize}
    \item $\mathrm{embdim}(R)-\depth(R)\les 3$;
    \item $R$ is one link from a complete intersection ring;
    \item   $R$ is a Gorenstein ring that is codimension 4 or two links from a complete intersection;
    \item $R$ is a
    Koszul algebra. 
\end{itemize}
The first three cases were established in \cite{Avramov:1989a} and the Koszul case was shown in \cite{Lofwall:1994}. In the more than 30 years since these results were established, there has been little progress on Avramov's question. Our main result says that the answer to \cref{Lucho_question} is ``yes" when $R$ is defined by monomials.

\begin{introthm}
\label{main_thm_intro}
    Let $Q/I$ be a minimal regular presentation for $R$, and assume that $I$ is minimally generated by monomials in some regular sequence of $Q$. The following are equivalent:
\begin{enumerate}
    \item \label{intro_thm_1}The ring $R$ has an embedded deformation of codimension $c$.
    \item  \label{intro_thm_2}The conormal module $I/I^2$ of $R$ has a free summand of rank $c$. 
    \item  \label{intro_thm_3}There is a $k$-subspace of rank $c$ in $\pi^2(R)$ that consists of central elements of $\pi(R)$.
    \item  \label{intro_thm_4}There is a $k$-subspace of rank $c$ in $\pi^2(R)$ that consists of radical elements of $\pi(R)$.
\end{enumerate}
\end{introthm}

Recall that an element $z$ in a Lie algebra $L$ is radical if $\ad(z)^n$ is identically zero on $L$ for some $n>0$. \cref{main_thm_intro} can be found in \cref{monomial_equivalence}, where more is shown: if any of the equivalent conditions above hold, then equivalently $I$ is generated by monomials $\f$ in a regular sequence $\bm{x}$ of $Q$ where, up to reordering, 
\[
\supp(f_i)\cap\supp(f_j)=\varnothing \quad\text{  for each }1\les i\les c\text{ and all }j\neq i\,,
\]
where the support of a monomial records those $x_i$ that divide it.

The implications $\cref{intro_thm_1}\Rightarrow\cref{intro_thm_2}\Rightarrow\cref{intro_thm_3}\Rightarrow\cref{intro_thm_4}$, from \cref{main_thm_intro}, always hold; see \cref{th_radical_support}. The fact that the implication $\cref{intro_thm_4}\Rightarrow \cref{intro_thm_1}$ holds in the setting of \cref{main_thm_intro} is a bit unexpected for a number of reasons. For one, it forces central elements to coincide with radical elements in $\pi^2(R)$. Furthermore, even in the previous situations that \cref{Lucho_question} was known to have a positive answer, the only case where $\cref{intro_thm_4}\Rightarrow \cref{intro_thm_1}$ was known to hold is when $\pi(R)$ can be computed in its entirety (for example, when $\mathrm{embdim}(R)-\dim(R)\les 3$, or trivial situations like when $R$ is complete intersection or Golod). Part of the novelty of \cref{main_thm_intro} is that it proves this stronger implication without having to compute $\pi(R)$, which is typically an infinite calculation. 

The proof of \cref{main_thm_intro} exploits a relationship,  established by the present authors in \cite{BEJ2}, between radical elements in $\pi^2(R)$ and hyperplane sections containing the cohomological support variety of $R$. Recall that the cohomological support variety of an $R$-complex $M$ with finitely generated homology is a conical variety in affine $n$-space over $k$, where $n=\rank\pi^2(R)$. These are interesting geometric objects that detect homological properties of the complex or ring; cf.\@ \cite{Avramov/Buchweitz:2000b,Briggs/Grifo/Pollitz:2022,BEJ2,Pollitz:2019,Stevenson:2014a}. The connection of $\V_R(R)$ to the structure of $\pi(R)$ was established in \cite[Theorem~3.1]{BEJ2}: \emph{if there is a $c$-dimensional subspace of $\pi^2(R)$ that consists of radical elements of $\pi(R)$ then $\V_R(R)$ is contained in a codimension $c$ hyperplane of $\mathbb{A}_k^n$}. This result is paired with \cref{main thm monomial}, which gives a criterion for $\V_R(R)$ to be contained in a hyperplane of a given codimension in the monomial setting. 

 The proof of \cref{main thm monomial} goes by an analysis of the Taylor resolution of $R$ over the regular ring. Such analysis also led to the next two main results. First, recall that the \emph{realizability question} of cohomological supports is the following: \emph{for an arbitrary local ring $R$ with $n=\rank\pi^2(R)$,
is every conical subvariety of $\mathbb{A}_k^n$ the cohomological support variety of some bounded complex of finitely generated $R$-modules?} The answer is ``yes" when $R$ is complete intersection; see \cite{Bergh}, as well as \cite{Avramov/Iyengar:2007}.  Because of certain non-realizability results in \cite[Theorem B]{Pollitz:2021} or \cite[Theorem~3.1]{BEJ2}, the answer is expected to be ``no" whenever $R$ is not complete intersection; see \cref{conjecture_realizability}. In this article, we show the following---this is \cref{cid_monomial}---in the monomial setting:

\begin{introthm}
     Let $Q/I$ be a minimal regular presentation for $R$, and assume that $I$ is minimally generated by $n$ monomials in some regular sequence of $Q$. For each $R$-complex $M$ with $\H(M)$ nonzero and finitely generated, there is an inequality 
     \vspace{-0.5em}
     \[\dim\V_R(M)\geqslant n - \operatorname{height}(I) \,. \]
     In particular, $\dim \V_R(M)>0$ when $R$ is not complete intersection. 
\end{introthm}

The first inequality was established for Cohen-Macaulay rings in \cite[Corollary~2.8]{BEJ2}. The lower bound in the theorem is called the complete intersection defect of $R$ and is discussed more in \cref{cid_defect}; the relevant point to the theorem above is that it is zero exactly when $R$ is complete intersection.

Finally, we classify all varieties that can be realized as $\V_R(R)$ where $R$ is defined by up to five monomials.

\begin{introthm}\label{introthm_mon_support}
   Let $R$ have regular presentation $Q/I$ where $I$ is generated by monomials $\f=f_1,\ldots,f_n$ on a regular sequence $\bm{x}$ of $Q$. 
Then there are only finitely many closed subsets of $\mathbb{A}_k^n$ that are realizable as $\V_R(R)$, depending only on $n$, and independent of $\bm{x}$ or $Q$. 
If $n\leqslant 5$, then the cohomological support variety of $R$ is a union of coordinate subspaces. In particular, there are two cases: 
\begin{itemize} 
\item (the exceptional case): if $n=5$ and, up to reordering, $\f$ has GCD graph
\begin{center}
\begin{minipage}{0.3\textwidth}
\centering
\vfill
\begin{tikzpicture}[scale=.8, every node/.style={}]
  \node (v1) at (0,0) {1};
  \node (v2) at (1.5,0) {2};
  \node (v3) at (3,0) {3};
  \node (v4) at (4.5,0) {4};
  \node (v5) at (6,0) {5};

  \draw (v1) -- (v2);
  \draw (v2) -- (v3);
  \draw (v3) -- (v4);
  \draw (v4) -- (v5);
\end{tikzpicture}
\vfill
\end{minipage}
\hspace{.4in}
\begin{minipage}{0.3\textwidth}
\centering
\begin{tikzpicture}[scale=.65, every node/.style={}]
  \node (v1) at (0,-.25) {1};
  \node (v2) at (1.5,-.25) {2};
  \node (v3) at (2.25,.8) {3};
  \node (v4) at (3,-.25) {4};
  \node (v5) at (4.5,-.25) {5};

  \draw (v1) -- (v2);
  \draw (v2) -- (v3);
  \draw (v2) -- (v4);
  \draw (v3) -- (v4);
  \draw (v4) -- (v5);
\end{tikzpicture}
\end{minipage}
\end{center}
and $f_3$ divides $\lcm(f_2,f_4)$, then 
\[
\V_R(R)=\{a\in \mathbb{A}_k^5\mid a_1=0\text{ or } a_5=0\}\,.
\]
\item (the generic case): in all other cases, 
\[
\V_R(R) = \left\{ a\in \mathbb{A}_k^n \mid  a_i=0\text{ for all }i\in S\right\}
\]
where $S\subseteq\{ 1,\ldots, n\}$ and $c= |S|$ is the largest codimension of an embedded deformation of $R$.
\end{itemize}
\end{introthm}

The first part of the theorem is \cref{support_monomial}, while the classification results are in \cref{th_nleq4,classification of n=5}.  The main approach used to establish \cref{introthm_mon_support} is studying two combinatorial objects associated to the defining relations of $R$, the GCD graph and Taylor graph; see \cref{gcd_graph,definition_Taylor_graph} for precise definitions.

In the introduction, we have focused on the local setting, where Avramov's question was originally asked; however, we simultaneously work in the graded context in this paper. That is to say, in the paper, by a minimal regular presentation $Q/I$ of $R$ we mean that one of the following holds:
\begin{itemize}
    \item (\textbf{Local case}): $\widehat{R}=Q/I$ with $(Q,\m)$ a regular local ring and $I\subseteq \m^2$, or
    \item (\textbf{Graded case}): $R=Q/I$ where $Q$ is a positively graded polynomial algebra over a field and $I$ an ideal generated by homogeneous forms of degree at least 2. 
\end{itemize}
The theorems above all hold in this generality. It is worth remarking that Avramov's question has a negative answer in the nonstandard graded case due to an example of Dupont~\cite{Dupont}, provided one interprets Avramov's question in this setting as the ring needs to admit a \emph{homogeneous} embedded deformation; see \cref{example_Dupont} for more on this. In fact, Dupont's example is analyzed again in \cref{example_Dupont}, where we show if we give the ring the standard grading then the ring admits a homogeneous embedded deformation, as does the corresponding local ring. 

Despite the evidence in the references above and the results below, while preparing this article, in collaboration with Mark E.~Walker, we have discovered a counterexample to \Cref{Lucho_question} in the local setting that will appear in future work. 

\section{Background}
Throughout, $Q$ is a commutative base ring that is either local or a nonnegatively finitely generated graded algebra over a field $k=Q_0$. Let $\m$ denote the unique maximal ideal of $Q$, in the local case, or the irrelevant maximal ideal of $Q$ in the graded setting. In either case, $k=Q/\m$ is the residue field of $Q$. 

\subsection{Semifree dg algebras and the homotopy Lie algebra}\label{s_dg}

We start by recalling the necessary background on semifree differential graded (=dg) algebras and the homotopy Lie algebra. The reader is encouraged to refer to \cite{Avramov:2010}, and the references therein, for a more detailed account of this material. 

\begin{ch}
    Given a dg $Q$-algebra $A$, we write $A^\natural$ for its underlying graded $Q$-algebra. A \textbf{semifree (algebra) extension} of $A$ is a dg $A$-algebra $B$ that is isomorphic as a graded $A^\natural$-algebra to the free strictly graded commutative $A^\natural$-algebra on some set of graded variables $X=X_1\cup X_2\cup \cdots $, where $X_i$ consists of variables of homological degree $i$. That is to say, there is an isomorphism of graded $A^\natural$-algebras
    \[
    B^\natural\cong \Lambda_{A^\natural}(X_{\text{odd}})\otimes_{A^\natural} \sym_{A^\natural}(X_{\text{even}})\,.
    \]
    In this case, we write $B=A[X]$. When $X$ consists of finitely many variables $x_1,\ldots, x_n$ with $\partial(x_i) = a_i$,
    \[B=A[x_1,\ldots,x_n\mid \del x_i=a_i]\,.\] 
\end{ch}

\begin{ex}
\label{e_koszul}
    Given $\f=f_1,\ldots,f_n$ in $Q$, then $E=Q[e_1,\ldots,e_n\mid \del e_i=f_i]$ is the Koszul complex on $\f$.
\end{ex}

\begin{ch}
\label{c_semifree_alg} 
Fix a quotient ring map $Q\to R$.
A \textbf{semifree resolution} of $R$ over $Q$ is a quasi-isomorphism of dg $Q$-algebras $Q[X]\to R$ where $Q[X]$ is semifree over $Q$. In the graded case, the variables in $X$ also have an internal grading. 
One can inductively construct such a resolution using a method of Tate that consists of successively adjoining variables, minding parity to determine whether an exterior or symmetric variable should be added, to kill homology classes; see \cite{Tate:1957} or \cite{Avramov:2010} for details. 

If one constructs such a resolution $Q[X]$ where the minimal number of variables is used in each homological degree, the resulting dg-algebra $Q[X]$ is unique up to isomorphism~\cite[Proposition 7.2.4]{Avramov:2010} and has a {\bf decomposable} differential~\cite[Lemma 7.2.2]{Avramov:2010}, meaning $\partial(X)\subseteq (\m, X^2)$. In this case, we call $Q[X]$ the \textbf{minimal model of $R$ over $Q$}. This will be pertinent to the construction of the homotopy Lie algebra in \cref{HLA}.\looseness -1
\end{ch}

\begin{ch}
    \label{c_deviations}
Assume $R=Q/I$, or in the local case that $\widehat{R}\cong Q/I$, with $Q$ regular and $I\subseteq \m^2$. In the graded case, we assume that $I$ is a homogeneous ideal. In this setting, we refer to a minimal model $Q[X]$ of $R$ over $Q$ as \emph{the} minimal model of $R$. This is independent of $Q$ whenever $R$ is an equicharacteristic ring, and regardless the semifree dg $k$-algebra
\vspace{-0.5em}
\[k[X]=k\otimes_Q Q[X]\]
is unique up to isomorphism; see \cite[Proposition~3.1]{Avramov:1999}.
In particular, the number of variables adjoined in each degree is always independent of $Q$.

The {\bf deviations} of $R$ are defined as follows: we set 
$$\varepsilon_i(R) = |X_{i-1}| \quad \text{ for } i \geqslant 2$$ 
and by convention $\varepsilon_1(R)= \dim(Q)$ is the embedding dimension of $R$. In light of \cref{e_koszul}, the first step $Q[X_1]$ towards constructing the minimal model $Q[X]$ is the Koszul complex on a minimal generating set for $I$, and so $\varepsilon_2(R)$ is the minimal number of generators of $I$.
\end{ch}

\begin{ch}\label{HLA}
    Fix the minimal model $Q[X]$ of $R$ as defined above.   Following \cite{Avramov/Halperin:1987}, the \textbf{homotopy Lie algebra of} $R$ is the graded $k$-vector space $\pi(R)=\{\pi^i(R)\}_{i\geqslant 1}$ given by
\[
\pi^i(R)\coloneqq \begin{cases}
(\m/\m^2)^\vee & i=1\\
(kX_{i-1})^\vee   &  i\geqslant 2
\end{cases}
\]
where $(-)^\vee$ denotes taking the $k$-linear dual. 
Elements of $\pi^i(R)$ can be regarded as graded $k$-linear functionals $kX\to k$.  Note that the rank of the $k$-vector space $\pi^i(R)$ is $\varepsilon_i(R)$, the $i^{\text{th}}$-deviation  of $R$. 

Furthermore, using that the differential $\del$ of $k[X]$ is decomposable, a Lie bracket can be defined on $\pi^{\geqslant 2}(R)$ in terms of the quadratic part of $\del$. Namely, given $\alpha \in \pi^i(R)$ and $\beta\in \pi^j(R)$, define $[\alpha,\beta]\in \pi^{i+j}(R)$ as the following degree $1-i-j$ functional on $kX$: if $x\in X_{i+j-1}$, we write 
\[\partial(x)=\sum_{pq} a_{pq} y_pz_q \quad\text{in } k[X]/(X)^3\,,\] 
with $a_{pq}\in k$ and $y_p,z_q\in X$, and then set
\[
[\alpha,\beta]( x) \coloneqq \sum_{pq} a_{pq}\big( (-1)^{i+1+ij} \alpha( y_p) \beta( z_q) +(-1)^{j} \beta(y_p) \alpha( z_q)\big)\,.
\]
 By \cite[page 175]{Avramov/Halperin:1987} the bracket defined above makes $\pi^{\ges 2}(R)$ into a graded Lie algebra over $k$; see \cite[Remark~10.1.2]{Avramov:2010} for the technicalities present in characteristic 2 or 3. 
\end{ch}

\begin{rem}
    The description above makes $\pi^{\geqslant 2}(R)$ into a graded Lie algebra over $k$; this is all we need in the present paper. However, one can extend the bracket to make $\pi(R)$ a graded Lie algebra where the full decomposability of the differential of $Q[X]$ is used to formulate $\pi^1(R)\times \pi^i(R)\to \pi^{i+1}(R)$; confer \cite{Briggs:2018}. We will also discuss another description of $\pi(R)$ in \cref{r_pi_as_derivations}.
\end{rem}

\begin{ch}
    \label{c_identification}
The homotopy Lie algebra of $R$ in degree 2 admits a simple description. Fix a minimal regular presentation $R=Q/I$, and recall that the Koszul complex $Q[X_1]$ is  the first step in the construction of the minimal model $Q[X]$; see \cref{c_deviations}. It follows that there is a canonical isomorphism
\[
    \pi^2(R) = (kX_1)^\vee  \xrightarrow{\ \cong\ } (I/\m I)^\vee \quad\text{given by } \ x^\vee\mapsto \del(x)^\vee\,.
\]
\end{ch}

\begin{ch}
For any element $\alpha \in \pi^i(R)$ we consider the adjoint action
\[
\ad(\alpha)\coloneqq[\alpha, -] \colon \pi^*(R) \to \pi^{*+i}(R)\,.
\]
We say that $\alpha$ is {\bf central} if $\ad(\alpha)=0$, and {\bf radical} if $\ad(\alpha)^n=0$ for some $n$. The \textbf{center} $\zeta(R)$ and \textbf{radical} $\rho(R)$ of $\pi(R)$ are 
\begin{align*}
\zeta(R) &\coloneqq \{ \alpha\in \pi(R) \mid \alpha\text{ is a central element of }\pi(R)\}\\
\rho(R) &\coloneqq \{ \alpha\in \pi(R) \mid \alpha\text{ is a radical element of }\pi(R)\}\, .
\end{align*}
We always have $\zeta(R)\subseteq \rho(R)$ and by  \cite{Felix/Halperin/Jacobsson/Lofwall/Thomas:1988}  $\rho(R)$ is a graded Lie ideal in $\pi(R)$; in fact, it is the maximal solvable ideal. We will be especially interested in radical elements living in degree $2$, which we will denote by $\rho^2(R)$. More generally, $\zeta^i(R)$ and $\rho^i(R)$ will denote the central and radical elements of degree $i$, respectively.
\end{ch}

\begin{ex}\label{ex_ci}
    Assume $R$ is complete intersection, meaning that $I$ is generated by a $Q$-regular sequence. In this case, 
    \[
   \zeta^2(R)=\rho^2(R)=\pi^2(R)=\pi^{\geqslant 2}(R)\,.
    \]
    In fact, by \cite[Theorem~C]{Avramov/Halperin:1987}, see also \cite[Corollary 3.4]{BEJ2}, the only time $\rho^2(R)=\pi^2(R)$ is when $R$ is complete intersection. 
\end{ex}

\begin{ex}[Golod rings]
Let $K^R$ denote the Koszul complex on a minimal generating set for the maximal ideal of $R$ and 
$$\codepth(R)=\max\{i: \H_i(K^R)\neq 0\}.$$
The following coefficientwise inequality on the Poincar\'e series of $\mathrm{P}_k^R(t)$ is due to Serre:
\[
\mathrm{P}_k^R(t)\preccurlyeq\frac{(1+t)^{\varepsilon_1(R)}}{1-t\sum_{i=1}^{\codepth R}\rank_k\H_i(K^R) t^i}\, .
\]
See for example \cite[Proposition~4.1.4]{Avramov:2010} for a proof of the inequality in general. When equality is acchieved, we say that the ring $R$ is Golod. See \cite{Golod} or \cite[Section~5]{Avramov:2010} for more on Golod rings. By \cite{Golod}, $\pi^{\geqslant 2}(R)$ is the free Lie algebra on $(\shift \H_{>0}(K^R))^\vee$. In particular, whenever $\codepth(R)>1$, it follows that $\zeta(R)=\rho(R)=0$. When $\codepth(R)\leqslant 1$, the ring $R$ is a hypersurface, and hence we are in the context of \cref{ex_ci}, so that $\pi^{\geqslant 2}(R) = \pi^2(R)$.
\end{ex}

\subsection{Cohomological support varieties}\label{s_support}
We will now recall the definition of cohomological support varieties over a local ring~\cite{Jorgensen:2002,Pollitz:2019,Pollitz:2021}; these extend the ones introduced to commutative algebra by Avramov in \cite{Avramov:1989}; see also \cite{Avramov/Buchweitz:2000b}. First, some background on cohomological actions. 

\begin{ch}
\label{c_semifree}
Let $A$ be a dg $Q$-algebra. We say a dg $E$-module $F$ is \textbf{semifree} if it admits a semifree filtration; that is to say, $F$ has an exhaustive filtration by dg modules
\[
0=F(-1)\subseteq F(0)\subseteq F(1)\subseteq \ldots \subseteq F
\]
such that each subquotient $F(i)/F(i-1)$ is a sum of shifts of $A$. For example, any bounded below dg $A$-module whose underlying graded $A^\natural$-module is a free graded $A^\natural$-module is necessarily semifree. 

Recall that for each dg $A$-module $M$, there exists a surjective quasi-isomorphism  of dg $A$-modules $F\xra{\simeq} M$ with $F$ semifree. We call such a map (or $F$ itself) a semifree dg $A$-module resolution of $M$ over $A$. As the functor $\Hom_A(F,-)$ preserves surjections and quasi-isomorphisms, it follows that any two semifree resolutions are unique up to homotopy equivalence. For more general details on semifree modules see \cite{Avramov:2010} as well as \cite[Chapter~6]{Felix/Halperin/Thomas:2001}, and in the context of semifree dg modules over a Koszul complex, see \cite[Section~2]{Avramov/Buchweitz:2000a}.
\end{ch}

\begin{ch}\label{finite_semifree}
    Let $A$ be a nonnegatively graded dg $Q$-algebra. Assume that $\H_0(A)$ is a commutative noetherian ring and that each $\H_i(A)$ is finitely generated over $\H_0(A)$. By \cite[Appendix~B.2]{Avramov/Iyengar/Nasseh/SatherWagstaff:2019}, for each dg $A$-module $M$ with $\H(M)$ finitely generated over $\H(A)$, there exists a semifree resolution $F\xra{\simeq} M$ over $A$ with 
\[
F^\natural\cong \bigoplus_{j=i}^\infty \shift^j (A^\natural)^{\beta_j}
\]
for some nonnegative integers $\beta_j $. For any $s\geqslant\sup\{j:\H_j(M)\neq 0\}$, the soft truncation below $s$, given by
\[
F'= \cdots \to 0 \to \coker \del^F_{s+1}\to F_{s-1}\to F_{s-2}\to \cdots\,,
\]
is a dg $A$-module that is quasi-isomorphic to $M$. Furthermore, when $M$ is perfect over $Q$, that is, 
\[
M\simeq (0\to P_m\to P_{m-1}\to \cdots\to P_i\to 0)
\]
with each $P_i$ finite rank free over $Q$, then for a suitably high value of $s$ the soft truncation of $F$ below $s$ is a dg $A$-module quasi-isomorphic to $M$ and perfect when regarded as a $Q$-complex. 
\end{ch}

\begin{ch}
    \label{c_ext}
    Let $A$ be a dg $Q$-algebra, and $M,N$ dg $A$-modules. Recall that $\Ext_A(M,N)=\{\Ext_A^i(M,N)\}_{i\in \mathbb{Z}}$ is the graded $Q$-module given by
    \[
    \Ext_A^i(M,N)=\H^i(\Hom_A(F,N))\quad \text{for each }i\in \mathbb{Z}\,,
    \]
    where $F\xra{\simeq}M$ is a semifree resolution of $M$ over $A$. As any two semifree resolutions are unique up to homotopy equivalence, cf.\@ \cref{c_semifree}, this module is independent of the choice of semifree resolution of $M$. Elements of $\Ext_A(M,N)$ are homology classes of $A$-linear chain maps $\alpha\colon F\to \shift^iN$, and so are denoted $[\alpha]$; here $\shift^iN$ is the suspension of $N$, which is given by 
    \[
    (\shift^iN)_n=N_{n-i},\quad \del^{\shift^iN}=(-1)^i\del^N,\quad \text{and } \ a\cdot \shift^im=(-1)^{|a|i} \shift^i(am)\,.
    \]

   Recall there is a well-defined composition pairing
   \[
   \Ext_A(M,N)\otimes_Q\Ext_A(L,M)\to \Ext_A(L,N)\quad\text{given by }([\alpha],[\beta])\mapsto  [\shift^{|\beta|}\alpha\circ \tilde{\beta}]
   \] 
   where $\tilde{\beta}$ is a lift of $\beta$ fitting into the commutative diagram 
   \begin{center}
       \begin{tikzcd}
            & \shift^{|\beta|}G\ar[d,"\simeq"]\\
           F\ar[r,swap,"\beta"]\ar[ur,"\tilde{\beta}"]& \shift^{|\beta|}M 
       \end{tikzcd}
   \end{center}
   with  $F\xra{\simeq}L$ and $G\xra{\simeq} M$ semifree resolutions over $A$. This pairing makes each $\Ext_A(M,M)$ a graded $Q$-algebra, and each $\Ext_A(M,N)$ a graded $\Ext_A(N,N)$-$\Ext_A(M,M)$ bimodule.
\end{ch}

\begin{rem}
\label{r_pi_as_derivations}
By a theorem of Gulliksen~\cite{Gulliksen:1968} and Schoeller~\cite{Schoeller:1967}, the minimal free resolution $F$ of $k$ over $R$ admits a dg algebra structure where the underlying algebra is a free divided power algebra. Let $D$ denote the collection of all $R$-linear derivations $F\to F$ that respect divided powers. Then $D$ is a subcomplex of the endomorphism complex of $F$, and there is an isomorphism of graded $k$-vector spaces
\begin{equation}\label{e_pi}
\H(D)\cong \pi(R)\,;
\end{equation}
see \cite[Theorem~10.2.1]{Avramov:2010}  Also, in \cite[Chapter~10]{Avramov:2010} it is shown that $\H(D)$ is a graded Lie algebra over $k$, with the bracket coming from the graded commutator, whose universal envelope is $\Ext_R(k,k)$. Furthermore, the isomorphism in \cref{e_pi} above is an isomorphism of graded Lie algebras.
\end{rem}

For the remainder of the section, assume that $R=Q/I$ is a minimal regular presentation, meaning that in addition to the previous assumptions we further assume that $Q$ is regular. Fix a list of minimal generators $\f=f_1,\ldots,f_n$ for $I$; recall that the number $n$ is $\varepsilon_2(R)$. Also, set
\[
E \coloneqq Q[ e_1,\ldots,e_n\mid \del e_i=f_i]
\quad
\text{ and }
\quad
\S\coloneqq k[\chi_1,\ldots,\chi_n]
\]
where each $\chi_i$ has cohomological degree 2. As $\H_0(E)=R$, the augmentation map $E\to R$ is a map of dg $Q$-algebras. Hence, each $R$-complex is a dg $E$-module via restriction of scalars. 

\begin{ch}
\label{c_cohsupp}
By \cite[Section~2]{Avramov/Buchweitz:2000a}, $\S$ can be identified with a graded $k$-subalgebra of $\Ext_E(k,k)$. In fact, $\S\subseteq \Ext_E(k,k)$ is a finite free extension of graded $\S$-modules; see \cite[Remark~3.2.6]{Pollitz:2019}. Hence for any dg $E$-module $M$, through the composition  product in \cref{c_ext}, the graded $k$-space $\Ext_E(M,k)$ is a graded $\S$-module. 
Moreover, for any  dg $E$-module $M$ with $\H(M)$ finitely generated over $R$, the graded $\S$-module $\Ext_E(M,k)$ is finitely generated; see \cite[Proposition~3.2.5]{Pollitz:2019}, as well as \cref{c_computing_support}. 

Given an $R$-complex $M$, we define the \textbf{cohomological support variety} of $M$ as
\[
\V_R(M)\coloneqq \supp_\S \Ext_E(M,k)=\{\p\in \spec \S: \Ext_E(M,k)_\p\neq 0\}\,,
\]
where $\spec \S$ denotes the homogeneous spectrum of $\S$. When $\H(M)$ is finitely generated as an $R$-module, it follows that $\V_R(M)$ is the Zariski-closed subset 
\[
\cV(\ann_\S\Ext_E(M,k))=\{\p\in \spec \S: \p\supseteq \ann_\S\Ext_E(M,k)\}\,.
\]
Note that $
\V_R(M)=\varnothing \text{ if and only if } M\simeq 0$. Furthermore, when $\H(M)$ is finitely generated over $R$ then $\V_R(M)\subseteq \{(\chi_1,\ldots,\chi_n)\}$ if and only if $\Ext_E^{\gg 0}(M,k)=0.$

We will denote the irrelevant maximal ideal $(\chi_1,\ldots,\chi_n)$ of $\S$ by $\bm{0}$.
\end{ch}

\begin{rem}
When $R$ is an arbitrary local ring, then $\V_R(M)$ is defined in terms of a minimal Cohen presentation $\widehat{R}\cong Q/I$. 
The fact that $\V_R(M)$ is independent of the choice of a Cohen presentation, and of the minimal generators for the defining ideal of $\widehat{R}$ in such a presentation, is dealt with in \cite[Theorem~6.1.2]{Pollitz:2021}; this also applies to the independence of the generators for the setting above. 
\end{rem}

\begin{ex}
    As $\Ext_E(k,k)$ is finite free over $\S$, it follows that $\V_R(k)=\spec \S$; see \cref{c_cohsupp}. 
\end{ex}

\begin{ex}
\label{ex_characterization}
If $R$ is complete intersection, then $E\xra{\, \simeq \,} R$ and thus $\Ext_E(R,k)=k$, so $\V_R(R)=\{\bm{0}\}$. By \cite[Theorem~3.3.2]{Pollitz:2019},  the converse holds: if $\V_R(R)=\{\bm{0}\}$, then $R$ is complete intersection. Furthermore, $R$ is  complete intersection if and only if $\V_R(M)=\{\bm{0}\}$ for some finitely generated $R$-module $M$; cf.\@  \cite[Theorem~6.1.6]{Pollitz:2021}.
\end{ex}

In this paper, we will be particularly interested in studying $\V_R(R)$ when $R$ is not complete intersection. 

\begin{ex}
If $R$ is Golod and not a hypersurface ring, then $\V_R(R)=\spec \S$ by \cite[Theorem~4.1]{BEJ2}.
\end{ex}

\begin{ex}
\label{ex_codepth}
Recall that $R$ has an embedded deformation if  $R \cong S/(f)$ for some local ring $S$ with $f$ an $S$-regular element in the square of the maximal ideal of $S$; cf.\@ \cref{s_deformations}. It was calculated in \cite[Theorem~6.3.5]{Pollitz:2021}, that whenever  $\codepth(R)\les 3$ we have
\[
\V_R(R)=\begin{cases}
\{\bm{0}\} & \text{if } R \text{ is complete intersection}\\
\spec \S & \text{if $R$ does not admit an embedded deformation}\\
\cV(\zeta) & \text{for some }\zeta\in \S^2, \text{otherwise}.
\end{cases}
\]
\end{ex}

\begin{ex}
\label{lucho_example}
An example of a non-linear support $\V_R(R)$ was given in \cite[Example~5.4]{Briggs/Grifo/Pollitz:2022}: the ideal $I=(x^2, xy, yz, xw, w^2)$ in $Q = k[x,y,z,w]$ defining the ring $R=Q/I$ considered by Avramov in \cite[2.2]{Avramov:1981}, we have
\vspace{-0.5em}
\[
\V_R(R) = \cV(\chi_1\chi_5) = \cV(\chi_1)\cup \cV(\chi_5) \subseteq \spec k[\chi_1,\ldots,\chi_5]\,.
\]
\end{ex}

\vspace{0.5em}

\section{Computing cohomological support varieties}\label{s_computing_support}

In this section we fix a minimal regular presentation $R=Q/I$, and let $k$ denote the residue field of $R$. We also fix an algebraic closure $K$ of $k$.

Let $\f=f_1,\ldots,f_n$ minimally generate $I$, and set
\[
E \coloneqq Q[ e_1,\ldots,e_n\mid \del e_i=f_i]\,.
\]
Also, define
$\S\coloneqq k[\chi_1,\ldots,\chi_n]$,
where each $\chi_i$ has cohomological degree 2.

\begin{chunk}\label{c_computing_support}
    Let $F$ be a dg $E$-module. Define the dg $\S$-module $\cC_E(F)$ as
    \[
    \cC_E(F)^\natural\coloneqq \S\otimes_k \Hom_Q(F,k)\quad\text{with}\quad\del^{\cC_E(F)}=1\otimes \del^{\Hom(F,k)}+\sum_{i=1}^n \chi_i\otimes e_i \,.
    \]
    As $\cC_E(F)$ is a dg $\S$-module its homology $\H(\cC_E(F))$ is a graded $\S$-module. 

  Fix a dg $E$-module $M$ with $\H(M)$ finitely generated over $R$. In this case, there is an isomorphism of graded $\S$-modules
  \[
  \Ext_E(M,k)\cong \H(\cC_E(F))
  \]
  where $F$ is any dg $E$-module quasi-isomorphic to $M$ that is also a bounded below complex of finite rank free $Q$-modules; cf.\@ \cite[3.2]{Pollitz:2021} and \cite[Sections 2 and 3]{Avramov/Buchweitz:2000a}. In fact, there exists such $F$ that is perfect over $Q$ (that is to say, a bounded complex of free $Q$-modules); see \cref{finite_semifree}. In that case, $\cC_E(F)$ is a finite rank free graded $\S$-module and hence, this gives a direct way to see that $\Ext_E(M,k)$ is finitely generated over $\S$.  
\end{chunk}

\begin{chunk}\label{dg_structure}
    Suppose $\bm{g} = g_1,\ldots,g_m$ is a list of elements in $Q$ with $I\subseteq(\bm{g})$. 
Fix  a $Q$-free resolution $A\xra{\simeq} Q/(\bm{g})$, where $A$ is a dg $Q$-algebra.
Writing 
\[
    f_i=\sum_{j=1}^m c_{ij} g_j \quad\text{with}\quad a_{ij}\in Q\,
\]
defines a morphism of dg $Q$-algebras $E\to A$ determined by
\[
   e_i\mapsto \sum_{j=1}^m c_{ij} a_j\quad\text{with}\quad a_j\in A_1\text{ such that } \ \del a_j=g_j\,. 
\]
In particular, the morphism above defines a dg $E$-module structure on $A$. When $A$ is minimal, then $\cC_E(A)$ can be identified with the complex of graded $\S$-modules 
\[
\cdots \to \shift^{-4}\S\otimes_k \Hom_Q(A_2,k)\xra{ \sum_{i,j} \chi_i\otimes c_{ij}a_j}\shift^{-2}\S\otimes_k \Hom_Q(A_1,k)\xra{ \sum_{i,j} \chi_i\otimes c_{ij}a_j} \S\to 0\,.
\]
\end{chunk}

\begin{example}
\label{ex_ext}
    For $k$, the Koszul complex $K^Q$ satisfies $K^Q\xra{\simeq} k$ as dg $E$-modules. Moreover, $e_iK^Q\subseteq \m K^Q$ by \cref{dg_structure}.  Hence, $\cC_E(K^Q)=\S\otimes_k \Hom_Q(K^Q,k)$ and hence, 
    \[
    \Ext_E(k,k)\cong \S\otimes_k \Hom_Q(K^Q,k)\,.
    \]
\end{example}

\begin{lemma}\label{l_tor_ind}
   Fix ideals $I_1$ and $I_2$ of $Q$, and let $R_i=Q/I_i$ and $E_i$ denote the Koszul complex on a minimal list of generators for $I_i$ over $Q$. Set $R=R_1\otimes_Q R_2$ and $E=E_1\otimes_Q E_2$. If $\Tor^Q_i(R_1,R_2)=0$ for $i>0$, then there is an isomorphism 
   \vspace{-0.5em}
   \[
   \Ext_E(R,k)\cong \Ext_{E_1}(R_1,k)\otimes_k \Ext_{E_2}(R_2,k)
   \]
   that is compatible with the isomorphism of graded $k$-algebras $\S\cong \S_1\otimes_k \S_2$, where $\S_i$ and $\S$ are the polynomial rings of cohomology operators in $\Ext_{E_i}(k,k)$ and $\Ext_E(k,k)$, respectively. 
\end{lemma}

\begin{proof}
    As $(I_1 \cap I_2) / I_1 I_2 \cong \Tor_{1}^Q(R_1,R_2)=0$, it follows that 
    \[
    \frac{I}{\m I}\cong \frac{I_1}{\m I_1}\oplus \frac{I_2}{\m I_2}\,,
    \]
     which induces an isomorphism of $k$-algebras $\S\cong \S_1\otimes_k \S_2$ and an isomorphism of dg $Q$-algebras $E\cong E_1\otimes_Q E_2.$ Now fixing free $Q$-resolutions $F_i\xra{\,\simeq\,} R_i$ with a dg $E_i$-module structure, it follows from $\Tor_{>0}^Q(R_1,R_2)=0$ that $F=F_1\otimes_Q F_2$ is a free $Q$-resolution of $R$ that has a dg $E$-module structure. In particular, there is an isomorphism of dg $\S$-modules:
    \[
  \cC_{E}(F)\cong \cC_{E_1}(F_1)\otimes_k \cC_{E_2}(F_2)
    \]
    where the right-hand side is considered an $\S$-module via the isomorphism $\S\cong \S_1\otimes_k \S_2.$ Now by applying \cref{c_computing_support} and the K\"unneth formula we obtain the desired isomorphism of graded $\S$-modules:
    \[
    \Ext_E(R,k)\cong \Ext_{E_1}(R_1,k)\otimes_k \Ext_{E_2}(R_2,k)\,.\qedhere
    \]
\end{proof}
 
\begin{chunk}
\label{c_support_using_hypersurfaces}
Let $Q \to \tilde{Q}$ be a flat local extension which induces the inclusion on residue fields $k\subseteq K$, and for each $R$-complex $M$ set $\tilde{M}=M\otimes_Q \tilde{Q}$. For each $a\in \mathbb{A}_K^n$, fix a lift $\tilde{a_i}$  of $a_i$ to $\tilde{Q}$, and let $Q_a=\tilde{Q}/(\sum \tilde{a_i}f_i)$. 
If $\H(M)$ is finitely generated over $R$, 
\[\Ext_{Q_a}^i(\tilde{M},K) \neq 0 \text{ for all } i \gg 0 \quad \text{if and only if} \quad \pdim_{Q_a}(\tilde{M})=\infty\,.\]
By \cite[Theorem~5.2.4]{Pollitz:2021} and the Nullstellensatz, the variety $\V_R(M)$ is determined by and determines the following conical variety in affine space $\mathbb{A}_K^n$: 
\[
\{a=(a_1,\ldots,a_n)\in \mathbb{A}_K^n: \pdim_{Q_a}(\tilde{M})=\infty \}\cup\{0\}\,.\]
Hence, $\V_R(M)$ records those directions in $\mathbb{A}_K^n$ in which $M$ has infinite projective dimension. 
Since the closed subsets in $\spec \S$ and $\mathbb{A}_K^n$, corresponding to the cohomological support of $M$, determine one another uniquely we will freely identify these objects.
\end{chunk}

The next proposition is an immediate consequence of \cref{l_tor_ind}. 

\begin{proposition}\label{tor independent gives union of varieties}
Let $R_1=Q/I_1$, $R_1=Q/I_1$, and $R=R_1 \otimes_Q R_2$. If $\Tor^Q_i(R_1,R_2)=0$ for $i>0$, then the following diagram commutes:
\vspace{-0.7em}
\[
    \begin{tikzcd}
    \A_K^{\varepsilon_2(R)}\ar[r,"\cong"] & \A_K^{\varepsilon_2(R_1)}\times_K   \A_K^{\varepsilon_2(R_2)} \\
        \V_R(R)\ar[u,"\subseteq"] \ar[r,"\cong"] & \V_{R_1}(R_1)\times_K  \V_{R_2}(R_2)\ar[u,swap,"\subseteq"] 
    \end{tikzcd}
\]
\end{proposition}
 
\begin{chunk}
\label{c_2_periodic}
    Let $M$ be an $R$-complex with $\H(M)$ finitely generated. 
    There is a bounded complex of finite rank free $Q$-modules $F$ that is quasi-isomorphic to $M$ as a dg $E$-module. For any $a=(a_1,\ldots,a_n)\in \mathbb{A}_K^n$, the $K$-complex $\cC_{E_a}(F\otimes_Q \tilde{Q})$ is isomorphic to the following 2-periodic $K$-complex in sufficiently high degrees: 
\[\cdots  \longrightarrow \Hom_Q(F_{\rm even},k)\otimes_k K\xra{\ d_a\ }  \Hom_Q(F_{\rm odd},k)\otimes_k K\xra{\ d_a\ }
        \Hom_Q(F_{\rm even},k)\otimes_k K \longrightarrow \cdots 
    \]
    where $E_a=\tilde{Q}[e_a\mid \del e_a=\sum \tilde{a_i}f_i]$, and $d_a\coloneqq\del^{\Hom(F,k)}\otimes 1+\sum e_i\cdot \otimes a_i$.
    The 2-periodic complex above and its $K$-linear dual are exact simultaneously, and so we consider
    \[
    \widehat{\cC}_{E_a}(F)\coloneqq\cdots \longrightarrow  F_{\rm even}\otimes_Q  K\xra{\ d_a\ }    F_{\rm odd}\otimes_Q K\xra{\ d_a\ }
        F_{\rm even}\otimes_Q K\longrightarrow \cdots 
    \]
    where
    \[
    d_a\coloneqq\del^{F}\otimes 1+\sum e_i \otimes a_i.
    \]
\end{chunk}

The following is well-known to the experts, but also follows from \cref{c_support_using_hypersurfaces} and \cref{c_2_periodic}, and so we record it here for completeness; see also \cite{Avramov/Buchweitz:2000b} as well as \cite{Avramov/Iyengar:2018}.

\begin{proposition}\label{prop_supp_Chat}
For an $R$-complex $M$ with $\H(M)$ finitely generated, fix a dg $E$-module $F$ that is a bounded complex of finite rank free $Q$-modules and quasi-isomorphic to $M$ as a dg $E$-module. Then
    \[
    \V_R(M)=\{a\in \mathbb{A}_K^n: \H(\widehat{\cC}_{E_a}(F))\neq 0\} \cup\{0\}=\{a\in \mathbb{A}_K^n: \H_p(\widehat{\cC}_{E_a}(F))\neq 0\} \cup\{0\}\,,
    \]
    where $p\in \{\mathrm{even},\mathrm{odd}\}$. 
\end{proposition}
\begin{proof}
    Let $a\in \mathbb{A}_K^n$. Then 
    \begin{align*}
        a\notin \V_R(M)&\iff \Ext_{Q_a}^{\gg 0}(\tilde{M},K)=0 & \text{by \cref{c_support_using_hypersurfaces}}\\
        &\iff \H^{\gg 0}(\cC_{E_a}(F\otimes_Q \tilde{Q}))=0 & \text{by \cref{c_computing_support} over } Q_a \\
        &\iff \H_{p}(\widehat{\cC}_{E_a}(F))=0\quad \text{ for }p\in \{\text{even},\text{odd}\}\,.
    \end{align*}
    The final equivalence holds as $\Hom_K(\cC_{E_a}(F\otimes_Q \tilde{Q}),K)$ in high degrees agrees with $\widehat{\cC}_{E_a}(F)$; cf.\@ \cref{c_2_periodic}.
    \end{proof}

\begin{example}\label{e_homotopies}
Assume  $R= Q/I$ with $I$ minimally generated by $f_1,f_2$, and for notational ease, further assume that $Q$ is equicharacteristic and $k$ is algebraically closed, so $k=K$. Since $Q$ is a UFD, we can write $f_i=f_i'g$ for some $Q$-regular sequence $f_1',f_2'$, where $g=1$ when $R$ is complete intersection and $g\in \m$ otherwise. The minimal $Q$-free resolution $F$ of $R$ is
  \[
    0 \longrightarrow
    Q\xra{\begin{pmatrix} -f_2'\\f_1' \end{pmatrix}}
    Q^2 \xra{\begin{pmatrix} f_1 & f_2\end{pmatrix}}
    Q \longrightarrow
    0
  \]
  and it supports a dg $E$-module structure with the $e_i$ action given by:
\[   e_1\cdot : \  0 \longleftarrow
    Q\xleftarrow{\begin{pmatrix} 0 & g \end{pmatrix}}
    Q^2 \xleftarrow{\displaystyle \begin{pmatrix} 1 \\ 0 \end{pmatrix}}
    Q \longleftarrow
    0  \quad \text{and} \quad   e_2\cdot: \  0 \longleftarrow
    Q\xleftarrow{\begin{pmatrix} -g & 0 \end{pmatrix}}
    Q^2 \xleftarrow{\displaystyle \begin{pmatrix} 0 \\ 1 \end{pmatrix}}
    Q \longleftarrow
    0\,.  \]

 Now for any $(a_1,a_2)\in k$, the two-periodic complex $\widehat{C}_{E_a}(F)$ is 
 \[
 \cdots \longrightarrow k^2\xra{\begin{pmatrix}
     a_1 & 0 \\
     a_2 & 0
 \end{pmatrix}} k^2\xra{\begin{pmatrix}
     0 & 0 \\
     -a_2\overline{g} & a_1\overline{g}
 \end{pmatrix}} k^2\xra{\begin{pmatrix}
     a_1 & 0 \\
     a_2 & 0
 \end{pmatrix}}k^2 \xra{\begin{pmatrix}
     0 & 0 \\
     -a_2\overline{g} & a_1\overline{g}
 \end{pmatrix}} k^2\longrightarrow \cdots\,.
 \]
 In particular, whenever $(a_1,a_2)$ is nonzero, the complex $\widehat{C}_{E_a}(F)$ is exact if and only if $g=1.$ Therefore, if $R$ is not complete intersection then $\V_R(R)=\mathbb{A}_k^2$; cf.\@ \cref{ex_characterization}.
\end{example}

\begin{proposition}\label{linear_nzd}
If $x$ is an element of $\m\smallsetminus\m^2$ and $M$ is an $R$-module on which $x$ is a nonzero divisor, then there is a canonical isomorphism $\V_R(M)\cong\V_{R/(x)}(M/xM)$ of subvarieties of $\mathbb{A}^n_k$.
\end{proposition}

\begin{proof}
We may use the presentation $Q/(x)/(I+(x))$ for $R/(x)$ to compute $\V_{R/(x)}(M/xM)$. If $F\to M$ is a $Q$-free resolution with a dg $E$-module structure, then 
$F/xF\to M/xM$ is a $Q/(x)$-free resolution of $M/xM$, and it inherits a dg module structure over $\overline{E}=E/(x)E$. With this in hand, we notice that the complex $\widehat{\cC}_{E_a}(F)$ constructed above is identical with $\widehat{\cC}_{\overline{E}_a}(F/xF)$, so we are done by \Cref{prop_supp_Chat}.
\end{proof}

\

\section{The Taylor dg algebra}\label{s_taylor_tate}
Fix an ideal $I$ generated by $\f=f_1,\ldots,f_n$ in a unique factorization domain $Q$. 
When $Q$ is a polynomial ring over a field and $f_1, \ldots, f_n$ are monomials, there is a well-known construction of Taylor~\cite{Taylor} producing a dg algebra resolution for $Q/I$. We will use this construction in a more general setting.

\begin{notation}\label{signnotation}
Let $[n]$ denote the ordered set $\{ 1, \ldots, n \}$. All subsets $J = \{j_1, \ldots, j_s\} \subseteq [n]$ inherit an ordering $j_1 < \cdots < j_s$ from $[n]$.
The cardinality of $J$ is denoted by $|J|=s$. 
If $J$ and $L$ are disjoint subsets of $[n]$, we write
\vspace{-0.8em}
    \[
    {\sgn(J,L)}=(-1)^\varepsilon \quad \text{where}\quad \varepsilon = \left| \{(j,\ell) \,:\, j\in J,\ \ell\in L, \text{ and } j>\ell\}\right|.
    \]
When $J$ and $K$ are not disjoint, we set $\sgn(J,K) = 0$.
These signs satisfy 
    \[ \sgn(J,L)\sgn(L,J)=(-1)^{|J||L|} \vspace{-0.5em}\]
    and 
\[ \sgn(J,K)\sgn(J,L)=\sgn(J, K \cup L) \quad \text{for mutually disjoint } J, K, L \subseteq[n].\]
\end{notation}

\begin{construction}\label{taylor}
Let $Q$ be a UFD and let $f_1, \ldots, f_n \in Q$ be such that $I = (f_1, \ldots, f_n)$ is a proper ideal of $Q$.
We can choose a finite list of nonassociate irreducibles $x_1,\ldots, x_d$ so that each $f_j$ can be written as
\[
f_j = u_j x_1^{a_{j1}} \cdots x_d^{a_{jd}} \quad \text{ for some }a_{ji} \geqslant 0 \text{ and some } u_j \in Q^\times\,.
\]
Given a subset $J$ of $[n]$, set
\[
f_J\colonequals x_1^{a_{J1}}\cdots x_d^{a_{Jd}}\quad\text{where }a_{Ji}=\max\{a_{ji}:j\in J\}\,.
\]
In particular, $f_j=u_jf_{\{j\}}$. 

The \textbf{Taylor complex on $\f$} (with respect to $x_1,\ldots,x_d$) is a  free $Q$-complex $T=T(\f)$, defined as follows. As a free graded $Q$-module, $T^\natural =\bigwedge_Q (\shift Q^{\oplus n}),$ where we fix the basis in degree $i$ given by
\[
\{b_J: J\subseteq[n]\text{ with }|J|=i\}\,.
\]
The differential on $T$ is the $Q$-linear map determined by 
\[
\partial(b_J) = \sum_{i=1}^s (-1)^{i-1} \, \frac{f_J}{f_{J \smallsetminus \{ j_i \}}} \, b_{J \smallsetminus \{ j_i \}}\quad \text{with }J=\{j_1 < \cdots < j_s\}\,.
\]
We equip the Taylor complex with a $Q$-bilinear product, defined on basis elements by
\[
b_J\cdot b_{K} = {\sgn(J,K)} \,  \frac{f_J f_{K}}{f_{J\cup K}} \, b_{J\cup K}\,.
\]
Note that 
\[
b_J\cdot b_{K}=0 \quad \text{if } J \cap K \neq \varnothing.
\]
\end{construction}

\begin{remark}\label{remark product in taylor gcd}
Note that, up to a unit, ${f_J f_{K}}/{f_{J\cup K}}=\gcd(f_J, f_K)$.
\end{remark}

While Taylor's original construction is written for the case of monomials, the standard proof that $T$ is a dg algebra applies more generally \cite{Taylor}:

\begin{proposition}
The Taylor complex $T = T(\f)$ of \cref{taylor} is a dg $Q$-algebra with $\H_0(T)=Q/(\f)$.
\end{proposition}
\begin{proof}
  First, we check $T$ is a complex. To this end, observe that given $J=\{j_1, \ldots, j_s\}$ with $j_1<\cdots<j_s$ we have 
    \begin{align*}\del^2(b_J)&=\sum_{i=1}^s(-1)^{i-1} \frac{f_J}{f_{J \smallsetminus \{ j_i \}}} \, \del(b_{J \smallsetminus \{ j_i \}})\\
    &=\sum_{i=1}^s(-1)^{i-1} \, \frac{f_J}{f_{J \smallsetminus \{ j_i \}}}\left(\sum_{1\les \ell<i} (-1)^{\ell-1} \, \frac{f_{J\smallsetminus \{j_i\}}}{f_{J \smallsetminus \{ j_i,j_\ell \}}} \, b_{J \smallsetminus \{ j_i,j_\ell \}} +\sum_{i< \ell\les s} (-1)^{\ell} \,\frac{f_{J\smallsetminus \{j_i\}}}{f_{J \smallsetminus \{ j_i,j_\ell \}}} \, b_{J \smallsetminus \{ j_i,j_\ell \}} \right)  \\
    &=\sum_{1\les \ell< i \les s}
    (-1)^{i+\ell} \, \frac{f_J}{f_{J\smallsetminus \{j_i\}}}
    \frac{f_{J\smallsetminus \{j_i\}}}{f_{J \smallsetminus \{ j_i,j_\ell \}}} \, b_{J \smallsetminus \{ j_i,j_\ell \}}+ 
    (-1)^{i+\ell-1} \, \frac{f_J}{f_{J\smallsetminus \{j_\ell\}}}\frac{f_{J\smallsetminus \{j_\ell\}}}{f_{J \smallsetminus \{ j_i,j_\ell \}}}\, b_{J \smallsetminus \{ j_i,j_\ell \}}\\
    &=0.
    \end{align*}
    The last equality follows from
    \vspace{-0.5em}
    \[
    \frac{f_J}{f_{J\smallsetminus \{j_i\}}}
    \frac{f_{J\smallsetminus \{j_i\}}}{f_{J \smallsetminus \{ j_i,j_\ell \}}}=\frac{f_J}{f_{J\smallsetminus \{j_\ell\}}}\frac{f_{J\smallsetminus \{j_\ell\}}}{f_{J \smallsetminus \{ j_i,j_\ell \}}}.
    \]
This shows that $T$ is a complex. 
Moreover, since $\del(b_{\{i\}})=f_i$ up to unit multiple, then $\H_0(T)=Q/(\f)$.
   
    To prove that the product is associative, by $Q$-linearity it suffices to check that
    \[
    (b_J\cdot b_{K})\cdot b_{L} =  b_J\cdot (b_{K}\cdot b_{L})
    \]
    for all subsets $J,K,L\subseteq [n]$. If $J$, $K$, and $L$ are not pairwise disjoint, both sides are zero by definition, and the equality follows, so we may assume that $J$, $K$, and $L$ are pairwise disjoint. The desired equality becomes
    \[
    {\sgn(J,K)\sgn(J\cup K,L)} \, \frac{f_J f_{K} f_{J\cup K} f_L }{f_{J\cup K} f_{{J\cup K\cup L}}} = {\sgn(J,K\cup L)\sgn(K,L)} \, \frac{f_J  f_{K\cup L} f_{K} f_L }{f_{{J\cup K\cup L}} f_{K\cup L} }\,.
    \]
    Note that this does indeed hold, as sign rules of \cref{signnotation} give
    $${\sgn(J,K)\sgn(J\cup K,L)} = {\sgn(J,K\cup L)\sgn(K,L)}$$
    and two fractions simplify to the same.

    It only remains to check that the Leibniz rule holds, and for this there are three cases to consider. If $J\cap K$ contains at least two elements then by definition
    \[
    \partial(b_J\cdot b_{K})=\partial(0)=0 \quad \text{and}\quad \partial(b_J)\cdot b_{K} + (-1)^{|J|} b_J\cdot \partial(b_{K})=0.
    \]
    If there is exactly one element $j\in J\cap K$ then $\partial(b_J\cdot b_{K})=0$, while
\begin{align*}
    \partial(b_J)\cdot b_{K} + (-1)^{|J|} b_J\cdot \partial(b_{K}) &= \sgn( \{ j \},J\smallsetminus \{ j \}) \frac{f_J}{f_{J\smallsetminus \{j\}}} b_{J \smallsetminus \{ j \}} \cdot b_{K} \\
    &\hspace{10mm} + (-1)^{|J|}\sgn(\{j\},K\smallsetminus \{ j \})\frac{f_K}{f_{K\smallsetminus \{j\}}} b_{J}\cdot b_{K \smallsetminus \{ j \}}\\
    &= \sgn( \{ j \},J\smallsetminus \{ j \})\sgn(J\smallsetminus \{j\},K)\frac{f_J f_{J \smallsetminus \{ j \}}f_{K}}{f_{J \smallsetminus \{ j \}}f_{J\cup K\smallsetminus \{j\}}}  b_{J\cup K\smallsetminus \{j\}}\\
    &\hspace{10mm} +(-1)^{|J|}\sgn(\{j\},K\smallsetminus \{ j \})\sgn(J,K\smallsetminus \{j\})\frac{f_J f_{K}f_{K \smallsetminus \{ j \}}}{f_{J\cup K\smallsetminus \{j\}}f_{K \smallsetminus \{ j \}}}   b_{J\cup K\smallsetminus \{j\}}\\
    &= (-1)^{|J\smallsetminus\{j\}|}\sgn(J\smallsetminus \{j\},K\smallsetminus \{j\})\frac{f_J f_{J \smallsetminus \{ j \}}f_{K}}{f_{J \smallsetminus \{ j \}}f_{J\cup K\smallsetminus \{j\}}}  b_{J\cup K\smallsetminus \{j\}}\\
    &\hspace{10mm} +(-1)^{|J|}\sgn(J\smallsetminus\{j\},K\smallsetminus \{ j \})\frac{f_J f_{K}f_{K \smallsetminus \{ j \}}}{f_{J\cup K\smallsetminus \{j\}}f_{K \smallsetminus \{ j \}}}   b_{J\cup K\smallsetminus \{j\}}\\
    &=0\,.
\end{align*}
Here we have combined signs using the rules from \cref{signnotation}. 
Finally, we verify the Leibniz rule when $J\cap K=\varnothing$. In this case
\begin{align*}
    \partial(b_J\cdot b_{K}) &= {\sgn(J,K)} \partial\left(\frac{f_J f_{K}}{f_{J\cup K}}  \, b_{J\cup K}\right)\\
    &= {\sgn(J,K)} \sum_{j\in J\cup K} \sgn(\{j\},J\cup K) \frac{f_{J\cup K}f_J f_{K}}{f_{J \cup K\smallsetminus \{ j \}}f_{J\cup K}} \, b_{J\cup K\smallsetminus\{j\}}\\
    &= {\sgn(J\smallsetminus\{j\},K)} \sum_{j\in J} \sgn(\{j\},J) \frac{f_{J\cup K}f_J f_{K}}{f_{J \cup K\smallsetminus \{ j \}}f_{J\cup K}} \, b_{J\cup K\smallsetminus\{j\}}\\
    &\hspace{15mm}  + {\sgn(J,K)} \sum_{j\in K} \sgn(\{j\},K) (-1)^{|J|}\frac{f_{J\cup K}f_J f_{K}}{f_{J \cup K\smallsetminus \{ j \}}f_{J\cup K}} \, b_{J\cup K\smallsetminus\{j\}}\\
    &= \partial(b_J)\cdot b_{K} + (-1)^{|J|} b_J\cdot \partial(b_{K})\,,
\end{align*}
again making use of the sign tricks from \cref{signnotation}.
\end{proof}

\begin{example}\label{e_taylor_2}
When $n=2$, then $T = T(\f)$ can be taken to have the form
\[
    0 \to
    Q\xra{\begin{pmatrix} -f_2'\\f_1' \end{pmatrix}}
    Q^2 \xra{\begin{pmatrix} f_1 & f_2\end{pmatrix}}
    Q \to
    0
  \]
  where $f_i=df_i'$ for $d$ a greatest common divisor of $f_1,f_2$; cf.\@ \cref{e_homotopies}. In particular, note that $T$ is the minimal free resolution of $R.$ The dg algebra structure describe in \cref{taylor}, is the same as the one viewing the complex above as a Hilbert--Burch resolution (see, for instance, \cite[Example~2.1.2]{Avramov:2010}). 
  \end{example}
 
\begin{remark}
\label{taylor_resolution}
    When $\f$ is a $Q$-regular sequence, the Taylor complex $T$ on $\f$ is the Koszul complex and hence a free resolution of $Q/(\f)$ over $Q$. Moreover, whenever $\f$ is a list of monomials in a regular sequence, the complex $T$ is a free resolution of $Q/(\f)$ over $Q$; this is a well-known result of Taylor \cite{Taylor}. Note however that $T$ can be a resolution even when $\f$ is not of this form; see \cref{e_taylor_2} below. There are however many examples where $\H_i(T)\neq 0$ for some $i>0.$ A comprehensive study of the structure of $T$, at this level of generality, is reserved for another occasion.
\end{remark}

Below is an explicit example when $I$ is a monomial ideal.

\begin{example}
    Let $Q=k[x,y,z]$ and $I=(yz,xz,xy)$, then the Taylor complex for $I$ is the Taylor resolution:
	$$T\!: \, \xymatrix@C=5mm{0 \ar[r] & Q b_{123} \ar[rr]^-{\begin{pmatrix} 1 \\ -1 \\ 1 \end{pmatrix}} && Qb_{23} \oplus Q b_{13} \oplus Q b_{12} \ar[rrrr]^-{\begin{pmatrix}  0 & -x & -x \\ -y & 0 & y \\ z & z & 0 \end{pmatrix}} &&&& Q b_1 \oplus Q b_2 \oplus Q b_3 \ar[rrr]^-{\begin{pmatrix} yz & xz & xy \end{pmatrix}} &&& Q \ar[r] & 0.}$$
    Left multiplication by $b_1$ on $T$ is given by the following: 
    $$\xymatrix@C=5mm{
0 \ar[r] & Q b_{123} \ar[r] & Q b_{23} \oplus Q b_{13} \oplus Q b_{12} \ar[dd]^-{\begin{pmatrix} yz \\ 0 \\ 0 \end{pmatrix}} \ar[r] 
& Q b_1 \oplus Q b_2 \oplus Q b_3 \ar[r] \ar[dd]^-{\begin{pmatrix} 0 & 0 & 0 \\ 0 & 0 & y \\ 0 & z & 0 \end{pmatrix}} 
& Q \ar[dd]^-{\begin{pmatrix} 1 \\ 0 \\ 0 \end{pmatrix}} \ar[r] & 0 \\ \\ 
& 0 \ar[r] & Q b_{123} \ar[r] & Q b_{23} \oplus Q b_{13} \oplus Q b_{12} \ar[r] & Q b_1 \oplus Q b_2 \oplus Q b_3 \ar[r] & Q \ar[r] & 0.
}$$
So for example, $b_1 \cdot b_2 = z b_{12}$.
\end{example}

We will now adjoin variables (polynomial in even degrees and exterior in odd degrees) to the Taylor complex to construct a dg algebra resolution of $R$ over $Q$ \cite{Tate:1957,Avramov:2010}. In general, this resolution will be quite large, but as we will see in the next section, this resolution will be helpful for our purposes.

\begin{construction}[Taylor model]\label{taylor-tate}
    Let $Q$ is a commutative base ring that is either local or a nonnegatively graded algebra with $Q_0$ a field. Let $\m$ denote the unique maximal ideal of $Q$, in the local case, or the irrelevant maximal ideal of $Q$ in the graded setting. In either case, $k=Q/\m$ is the residue field of $Q$. 
Let $I$ be a proper ideal in $Q$, minimally generated by $f_1, \ldots, f_n$, and $R = Q/I$.

Fix a Taylor complex $T$ for $R$ over $Q$. A {\bf Taylor model} for $R$ over $Q$ is the complex $T[X]$ obtained from $T$ by adding polynomial variables in even degrees $i \geqslant 2$ and exterior variables in odd degrees $i \geqslant 3$ in the sense of Tate \cite{Tate:1957} to form a dg algebra resolution for $R$ over $Q$.
In more detail, $X = X_{\geqslant 2} = \bigcup_{i \geqslant 2} X_i$ is a set of variables, with $X_i$ a set of variables of degree $i$. For each $i \geqslant 1$, we fix cycles $z_1, \ldots, z_m$ whose homology classes minimally generate $H_i(T[X_{\leqslant i}])$, and take $X_{i+1} = \{ x_1, \ldots, x_m \}$, then we set
$$T[X_{\leqslant i+1}] = T[X_{\les i}][x_1, \ldots, x_m \mid \partial(x_j) = z_j].$$
Now define $T[X] \colonequals \operatorname{colim} T[X_{\les i}];$ in particular, Taylor models exist.
\end{construction}

\vspace{0.5em}

\section{Embedded deformations, support varieties, and the homotopy Lie algebra}\label{s_deformations}

Throughout, let $R$ be a local or positively graded ring with residue field $k$ and with a regular presentation $Q/I$. We write $\m$ for the (homogeneous) maximal ideal of $Q$, and we assume that $I\subseteq \m^2$.
We begin with some notation that will allow us to connect embedded deformations to elements of the homotopy Lie algebra, as well as to hyperplanes in the spectrum of the ring of cohomology operators.

We say that $R$ has an {\bf embedded deformation of codimension} $\mathbf{c}$ determined by $J$ and $\f$ if there exists an ideal $J \subseteq I$ such that, in the ring $S=Q/J$, the ideal $I/J$ is generated by a regular sequence $\f=f_1, \ldots, f_c$ of length $c$:
\vspace{-0.5em}
\[
I = J + (\f) \quad \text{and }\quad \f \text{ is regular on } Q/J.
\]
An embedded deformation of codimension $1$ determines a unique element $z_J$ in the $k$-vector space $ \Hom_Q(I,k)$ with 
\[
z_J(J)=0 \quad \text{and} \quad z_J(f)=1.
\]
Note that up to a scalar, $z_J$ depends only on $J$. By \Cref{c_identification} the homotopy Lie algebra in degree two is dual to $I/\m I$, and likewise $\S = k[\chi_1,\ldots,\chi_n]$ is the symmetric algebra on the dual of $I/\m I$, so there are canonical isomorphisms
\[
\pi^2(R) \xra{\ \cong\ } \Hom_Q(I,k)\xra{\ \cong\ } \S^{2}.
\]
Making this identification, the fact that the quotient map $S {\relbar\joinrel\twoheadrightarrow}R$ is complete intersection implies that $z_J$ is a central element in $\pi(R)$ \cite{Avramov:1989a}. In a similar vein, when $J$ and $f$ determine an embedded deformation, the map $z_J\colon I\to k$ lifts to a map $I\to R$. There must then be a section $R\to I/I^2$ generating a free summand of the conormal module $I/I^2$. As a partial converse to this, Iyengar has proven that free $R$-module summands of $I/I^2$ give rise to central elements of $\pi(R)$ \cite{Iyengar:2001}.  

A support-theoretic obstruction to the existence of embedded deformations was given in \cite[Proposition 6.3.7]{Pollitz:2021}: if $R$ has an embedded deformation given by $J$ and $f$, then $\V_R(R)$ is contained in a hyperplane: $\V_R(R) \subseteq \ker (z_J)\subseteq I/\m I$. In \cite[Theorem 3.1]{BEJ2}, we showed that radical (and in particular, central) elements of $\pi^2(R)$ also give rise to hyperplanes containing $\V_R(R)$. This gives us the following:

\begin{theorem}\label{th_radical_support}
Let $R$ be either a noetherian local ring or a positively graded finitely generated algebra over a field. Consider the following conditions:
\vspace{-0.2em}
\begin{enumerate}
    \item The ring $R$ has an embedded deformation of codimension $c$.
    \item The space of central elements of degree $2$ in $\pi(R)$ has rank at least $c$.
    \item The space of radical elements of degree $2$ in $\pi(R)$ has rank at least $c$.
    \item The cohomological support variety of $R$ is contained in a linear subspace of codimension $c$.
\end{enumerate}
The implications $(1) \Rightarrow (2) \Rightarrow (3) \Rightarrow (4)$ always hold.
\end{theorem}

\begin{proof}  
If $R$ has an embedded deformation of codimension $c$, then $I/I^2$ has a free summand of rank $c$. By \cite[Main Theorem]{Iyengar:2001}, $\rank_k \zeta^2(R)\geqslant c$. This shows $(1)$ implies $(2)$.
By definition, every central element is radical, and thus $(2)$ implies $(3)$ is obvious. We proved $(3) \Rightarrow (4)$ in \cite[Theorem 3.1]{BEJ2}.
\end{proof}

\begin{remark}
The implication $(3) \Rightarrow (4)$ comes from a direct correspondence between radical elements of degree $2$ and linear subspaces containing $\V_R(R)$: under the isomorphism $\pi^2(R) \xra{\ \cong\ } \S^{2}$, radical elements of $\pi^2(R)$ are sent to linear polynomials vanishing on the subvariety $\V_R(R)\subseteq \spec \S$  \cite[Theorem 3.1]{BEJ2}. In other words, there is a canonical embedding
\[
\rad^2(\pi^*(R)) \hookrightarrow \left\{ \chi\in\S^{2} \mid \V_R(R)\subseteq \mathcal{V}(\chi) \right\}\,.
\]
In particular, if $\pi(R)$ contains a radical element then $\V_R(R)$ is contained in a corresponding hyperplane.
\end{remark}

In \cite[Problem 4.3]{Avramov:1989a}, Avramov asked whether the implication $(2) \Rightarrow (1)$ holds. 
As noted in the introduction, in future work with Mark E.~Walker we will present an example that shows $(2)\Rightarrow (1)$ need \textbf{not} hold in general.
Nevertheless, if the defining ideal of $R$ is generated by monomials on a regular sequence, then all of the conditions in \cref{th_radical_support} are equivalent; see \cref{monomial_equivalence} below. First, we need the following key lemma, which holds without the monomial assumption.

\begin{lemma}\label{key lemma}
    Let $Q/I$ be a minimal regular presentation of a local ring $R$, with $I$ minimally generated by elements $f_1, \ldots, f_n$ of $(Q,\m, k)$. The following hold:

    \begin{enumerate}
    \item If $\gcd(f_i,f_j)$ is not a unit for some $i\neq j$, then 
    \[
    \mathbf{e}_i ,\mathbf{e}_j\in \V_R(R),
    \]
    where $\mathbf{e}_1,\ldots, \mathbf{e}_n$ are the basis elements for $I/\m I$ determined by the images of $f_1, \ldots, f_n$. 
    \item If for all $i$ there exists $g \in I \smallsetminus (f_i, \m I)$ such that $\gcd(f_i, g)$ is not a unit, then 
\[
\mathrm{span}_k (\V_R(R)) = \mathbb{A}_k^n\,.
\]
    \item\label{keylemitem3} If there is an $i$ such that $\gcd(f_i,f_j)$ is not a unit for all $j$, then 
    \[
    \V_R(R)=\mathbb{A}^n_k.
    \]
    \end{enumerate}
\end{lemma}

\begin{proof}
Let $T[X]=T[X_{\geqslant 2}]$ denote a Taylor model for $R = Q/I$ over $Q$ on $f_1, \ldots, f_n$ from  \cref{taylor-tate}; recall that $T[X]_1=T_1$ has a $Q$-linear basis $\{b_1,\ldots, b_n\}$ with $\del(b_\ell)=f_\ell$.

First, we prove (1). Since the roles of $f_i$ and $f_j$ are interchangeable, it suffices to prove that $\mathbf{e}_i \in \V_R(R)$. 
Let $C = \widehat{\cC}_{E_{\mathbf{e}_i}}(T[X])$ as in \Cref{c_2_periodic}, write $d$ for its differential, and by a slight abuse of notation write $b_\ell$ for its image in $C$. 
In the notation above, the goal is to show that $C$ has nonzero homology. In particular, we will prove that $b_j$ is a cycle but not a boundary in $C$.

First, we check that $b_j$ is a cycle in $C$. By the definition of the differential $d = d_{\mathbf{e}_i}$ on $C$ in \cref{c_2_periodic} is given by
$$d(b_j) = \partial(b_j) + \sum_{\ell=1}^n \left( {\mathbf{e}_i} \right)_\ell b_\ell\cdot b_j = b_i \cdot b_j.$$
Recall from \cref{taylor} that
\vspace{-0.7em}
\[
b_i \cdot b_{j} = \sgn(i,j) \gcd(f_i, f_j) b_{ \{ i, j \} },
\]
but we assumed that $\gcd(f_j, f_i)$ is not a unit, so $b_i\cdot b_j = 0$ in $C$.

We now show that $b_j$ is not a boundary. First, since $f_1, \ldots, f_n$ form a minimal generating set for $I$, any cycle of degree $1$ in $T[X]$ corresponds to a minimal syzygy of $I$ and hence $\partial(T[X]_2) \subseteq \m T_1$. Thus
$$(\operatorname{im }d)_1 = k \cdot \left( \sum_{\ell = 1}^n \left( {\mathbf{e}_i} \right)_\ell b_\ell b_{\varnothing} \right) = k \cdot b_i b_{\varnothing} = k \cdot b_i.$$
In particular, $b_j$ is not a boundary. We conclude that $C$ is not exact, and thus by \cref{prop_supp_Chat} we have $\mathbf{e}_i \in \V_R(R)$. This proves (1).

Now we note that (2) is a consequence of (1). Indeed, by assumption, for a fixed $i$ one replace $f_j$ with $g$ for some $j\neq i$ and the set $f_1,\ldots,g,\ldots f_n$ is a minimal generating set for $I$. This new generating set satisfies the hypothesis of \Cref{key lemma}, so $\mathbf{e}_i \in \V_R(R)$. Note that the change of coordinates on $\mathbb{A}_k^n$ determined by the change of generating sets above fixes $\mathbf{e}_i$. 
Thus
\[
\mathrm{span}_k \V_R(R) \supseteq \mathrm{span}_k \{ \mathbf{e}_1, \ldots, \mathbf{e}_n \} = \mathbb{A}_k^n\,.
\]

Finally, we show (3) holds. Fix $a \in \A^n_k$ and set $C = \widehat{\cC}_{E_{a}}(T[X])$. As in the proof of (1), the hypothesis implies that $b_i$ is a cycle in $C$. Furthermore,
\[
(\operatorname{im }d_a)_1 = k \cdot  \left(\sum_{j = 1}^n a_j b_j\right)\cdot b_{\varnothing}  = k \cdot \sum_{j=1}^na_jb_j\,,
\]
and so whenever $a \neq \lambda\mathbf{e}_i$ for some nonzero $\lambda\in k$, it follows that $b_i$ is not a boundary in $C$. Thus,
\[
\{
a \in \mathbb{A}_k^n:a = \lambda\mathbf{e}_i\text{ for some }\lambda\neq 0\}\subseteq \V_R(R).
\]
By \cref{c_support_using_hypersurfaces}, we may assume that $k$ is infinite. Since $\V_R(R)$ is a Zariski closed subset of $\mathbb{A}_k^n$, it follows that $\V_R(R)=\mathbb{A}_k^n$, as claimed.
\end{proof}

\begin{example}
    Consider the ideal $I = (x^2, xy, yz, zw, w^2)$ in $Q = k \llbracket x,y,z,w \rrbracket$ from \Cref{lucho_example}, which satisfies the hypothesis of (2) but not (3) in \Cref{key lemma}. As we noted in \Cref{lucho_example}, the support of $R = Q/I$ is a proper subvariety whose $k$-linear span is all of  $\A^5_k$. 
\end{example}

\begin{theorem}\label{main thm monomial}
Let $R$ be either a local or positively graded ring with residue field $k$, with a regular presentation $Q/I$. Write $\m$ for the (homogeneous) maximal ideal of $Q$, and assume that $I\subseteq \m^2$ is minimally generated by $f_1, \ldots, f_n$ which are monomials in some regular sequence of $Q$. If $\V_R(R)$ is contained in a linear subspace of codimension $c$, then $R$ has an embedded deformation of codimension $c$. 

More precisely, if the images of $f_{i_1},\ldots,f_{i_c}$ in $I/\m I$ are not in $\V_R(R)$, then $R$ has an embedded deformation given by $J=(f_j: j\neq i_\ell$ for some $1\leqslant \ell \leqslant c)$ and $f_{i_1},\ldots,f_{i_c}$.
\end{theorem}

\begin{proof}
Let $\mathbf{e}_1,\ldots, \mathbf{e}_n$ be the basis for $I/\m I$ determined by the images of $f_1, \ldots, f_n$. 
If $\V_R(R)$ is contained in a hyperplane of codimension $c$, then any basis of $I/\m I$ has $c$ many elements outside of $\V_R(R)$; in particular, up to reordering, $\mathbf{e}_1,\ldots, \mathbf{e}_c$ do not belong to $\V_R(R)$.

By \Cref{key lemma}, $\gcd(f_i, f_j)$ is a unit for all $1 \leqslant i \leqslant c$ and all $j \neq i$. We claim that since the $f_i$ are monomials on a regular sequence of $Q$, then
\[ f_1, \ldots, f_c \text{ is regular on } Q/(f_{c+1}, \ldots, f_n).\]
Indeed, the monomial assumption implies that $\H_\ell(T(\f))=0$ for $\ell \geqslant 1$ \cite{Taylor}, and since $\gcd(f_i, f_j)$ is a unit for all $1 \leqslant i \leqslant c$ and all $j \neq i$, then 
\[
T(\f)\cong T'\otimes_Q T''\quad \text{where }T'=T(f_1,\ldots,f_c)\text{ and } T''=T(f_{c+1},\ldots f_n)\,.
\]
As $f_{c+1},\ldots,f_n$ are monomials on a regular sequence of $Q$, it follows that $T'\xra{\simeq} \overline{Q}\coloneqq Q/(f_{c+1},\ldots,f_n)$. Moreover, the fact that $\gcd(f_i,f_j)=1$ for all $1\les i<j\les c$ also implies $T' \cong \mathrm{Kos}^{ Q}(f_1, \ldots, f_c)$.
Thus we have the quasi-isomorphism below:
\[T(\f)\xra{\simeq}T'\otimes_Q \overline Q\cong\mathrm{Kos}^{ Q}(f_1,\ldots,f_c)\otimes_Q \overline Q \cong  \mathrm{Kos}^{\overline Q}(f_1,\ldots,f_c)\,;\]  
Finally, combining this with the fact that $\H_\ell(T(\f))=0$ for $\ell \geqslant 1$, we conclude that $f_1,\ldots,f_c$ is regular on $\overline Q$.
Thus $R$ has an embedded deformation of codimension $c$ determined by $J=(f_{c+1},\ldots, f_n)$ and $f_1, \ldots, f_c$. 
\end{proof}

\begin{remark}
In the proof of \cref{main thm monomial}, one can instead use \cite[Lemma~3]{Kaplansky:1962} to deduce that $f_1,\ldots,f_c$ is regular on $Q/(f_{c+1},\ldots,f_n)$. 
\end{remark}

\begin{corollary}
\label{monomial_equivalence}
Let $Q/I$ be a minimal regular presentation for $R$. Assume that $I$ is minimally generated by $f_1, \ldots, f_n$ which are monomials in some regular sequence of $Q$. The following are equivalent:
\begin{enumerate}
\setcounter{enumi}{-1}
    \item The ring $R$ has an embedded deformation of codimension $c$ defined by monomials.
    \item The ring $R$ has an embedded deformation of codimension $c$.
    \item The space of central elements of degree $2$ in $\pi(R)$ has rank at least $c$.
    \item The space of radical elements of degree $2$ in $\pi(R)$ has rank at least $c$.
    \item The cohomological support variety of $R$ is contained in a linear subspace of codimension $c$.
\end{enumerate}
\end{corollary}

\begin{proof} 
The implication $(0) \Rightarrow (1)$ is immediate.
We already know that $(1) \Rightarrow (2) \Rightarrow (3) \Rightarrow (4)$ hold from \cref{th_radical_support}. The implication $(4) \Rightarrow (0)$ also holds in this setting because of \Cref{main thm monomial}, the proof of which shows that the defining ideal of the resulting embedded defomation is generated by monomials.
\end{proof}

\Cref{main thm monomial} says that if a monomial ring has an embedded deformation, then there must be an embedded deformation presented by a subset of the given monomial generators. But the set of minimal monomial generators is very special, and this behavior does not hold more generally, as the next example shows.

\begin{example}\label{example_Dupont}
    Let $k$ be a field. In \cite{Dupont}, Dupont considers the example of $R = Q/I$ with $Q = k[a,b,c,d,e]$ and $I = (f_1, \ldots, f_7)$, where
    \[f_1 = a^2c+d^3 \qquad f_2 = b^2d+e^3 \qquad f_3 = d^2e^2 \qquad
    f_4 = c^2 \qquad f_5 = cd \qquad f_6= ce \qquad f_7 = ab.\]
If we let $J = (f_1,\ldots,f_6)$, so that $I = J + (f_7)$, then the element $z_J \in \pi^2(R)$ is central by \cite{Dupont}.  
    Moreover, a calculation using Macaulay2 \cite{M2}, via the \texttt{ThickSubcategories} package, shows that $\V_R(R)$ is in fact the hyperplane determined by $J$, that is
    \vspace{-0.5em}
    \[ \V_R(R) =\ker(z_J).\]
    This shows that $R$ has at most an embedded deformation of codimension $1$.  Dupont considers a nonstandard grading on $R$, and proves \cite{Dupont} that there is no embedded deformation of $R$ giving rise to the element $z_J$ \emph{which respects this grading}; in particular, $f_7$ is not regular on $J$. 
    
    It is worth pointing out, however, that $I$ is homogeneous with respect to the standard grading on $Q$, and that under that grading there is a graded embedded deformation: one may take
    \[ J' = (f_1, f_2 + cf_7, f_3 +e^2 f_7, f_4, f_5, f_6 )\]
    and $J'$ and $ f_7$ form an embedded deformation for $R$.
\end{example}

\begin{example}
    Consider a field $k$ and $R=Q/I$ with $Q = k \llbracket a,b,c,d \rrbracket$ and $I = (f_1, f_2, f_3, f_4, f_5)$, where
    \[ f_1 = c^2, \quad f_2 = cd, \quad f_3 = a^2c + d^3, \quad f_4 = b^2d + ad^2, \quad f_5 = ab.\]
    In \cite{Iyengar:2001}, Iyengar notes that $ab$ generates a free summand of $I/I^2$. Moreover, Macaulay2 computations show that $\V_R(R)$ is precisely the hyperplane determined by the image of $f_5 = ab$ in $I / \m I$. Iyengar asks \cite{Iyengar:2001} whether $R$ has an embedded deformation, and indeed the answer is yes: $f_5$ is regular modulo
    \[ J = (f_1, f_2, f_3 + df_5, f_4 + cf_5) .\]
\end{example}

\vspace{0.5em}

\section{The realizability problem for rings defined by monomial ideals}\label{s_realizability}

This section is concerned with the so-called \textbf{Realizability Problem}. To discuss this, we adopt the following notation for the remainder of the section:
assume that a local ring $R$ has minimal regular presentation $Q/I$, let $k$ denote its residue field, and fix a list of minimal generators $\f=f_1,\ldots,f_n$ for $I$. Set
\[
E \coloneqq Q[ e_1,\ldots,e_n\mid \del e_i=f_i]\quad\text{and}\quad \S\coloneqq k[\chi_1,\ldots,\chi_n]\,,
\]
where each $\chi_i$ has cohomological degree 2.
 
 The realizability problem is the following: \emph{Determine what closed subsets of $\spec \S$ can be realized as $\V_R(M)$ for some $R$-complex $M$ with $\H(M)$ finitely generated.} 
When $R$ is complete intersection, every closed subset of $\spec \S$ is realizable; see, for example, \cite{Avramov/Iyengar:2007,Bergh,Burke/Walker:2015}.
 In \cite{Pollitz:2021}, it was shown that if $R$ is not complete intersection, then $\{\bm{0}\}$ is never realizable as $\V_R(M)$ when $M$ is a finitely generated $R$-module. Building on this, in \cite[Corollary~2.8]{BEJ2}, we established that if $R$ is a Cohen-Macaulay ring but not complete intersection, then 
\vspace{-0.5em}
 \[
 \dim \V_R(M)\geqslant 2
 \]
 for each $R$-complex $M$ with $\H(M)$ nonzero and
 finitely generated. In particular over such rings, $\{\bm{0}\}$ is never realizable. Based on the evidence above, we propose the following: 

\begin{conjecture}
\label{conjecture_realizability}
The ring $R$ is complete intersection if and only if $\{\bm{0}\}$ is realizable as $\V_R(M)$ for some $R$-complex $M$ with $\H(M)$ finitely generated.
 \end{conjecture}

\begin{remark}
\label{cid_defect}
    Recall that the complete intersection defect of $R$, denoted $\cid(R)$, is defined as
    \[
    \cid(R) \colonequals \varepsilon_2(R)-\varepsilon_1(R)+\dim(R)=n-\height(I)\,.
    \]
See \cite{Avramov:1977,KiehlKunz} for more on this invariant. An important feature is that the complete intersection defect detects the complete intersection property: $R$ is complete intersection if and only if $\cid(R)=0$. 
    
A numerical strengthening of \cref{conjecture_realizability} in terms of this invariant is the following: \emph{$R$ is complete intersection if and only if $\dim \V_R(M)\geqslant \cid(R)$ for all $R$-complexes $M$ with $\H(M)$ nonzero and finitely generated}. We showed in \cite[Corollary~2.8]{BEJ2} in that this holds whenever $R$ is Cohen-Macaulay.
\end{remark}

\begin{theorem}\label{cid_monomial}
    With the notation above, if $\f$ is a list of monomials on a regular sequence of $Q$, then $\dim\V_R(M)\geqslant \cid(R)$ for all $R$-complexes $M$ with $\H(M)$ nonzero and finitely generated. In particular, \cref{conjecture_realizability} holds for such rings. 
\end{theorem}
 
\begin{proof}
Set $\Lambda=E\otimes_Q k$. For a dg $\Lambda$-module $N$, we write $\ll_{\Lambda}(N)$ for its Loewy length, the smallest nonnegative integer $t$ such that $\Lambda^{\ges t} N=0$.

Let $T$ be the Taylor dg algebra on $\f$ over $Q$ and set $A=T\otimes_Q k$. As $\f$ is a list of monomials on a regular sequence in $Q$, it follows from \cref{taylor_resolution} that $T\xra{\simeq}R$. The augmentation map $E\to R$ defines a dg algebra map $E\to T$ by \cref{dg_structure}. In particular, this induces a dg $k$-algebra map $\Lambda\to A$, and so $A$ will be regarded as a dg $\Lambda$-module via this map. 

We claim that $\ll_{\Lambda}(A) \leqslant \height(I)$. 
To this end, in the notation of \cref{taylor}, the image of $\Lambda^i \to A$ is the $\Lambda$-submodule of $A$ generated by  all products of the form: 
    \[
    b_{j_1} \cdot b_{j_2} \cdots b_{j_i}\quad \text{where } j_1<\cdots<j_i\,.
    \]
    Thus, to justify the claim it suffices to show that any such product is zero in $A$ whenever $i>\height(I)$. Indeed, as $i$ is larger than the $\height(I)$ and $Q$ is Cohen-Macaulay, it follows that $f_{j_1},\ldots,f_{j_i}$ is not a regular sequence. Furthermore, since $\f$ is a list of monomials on a regular sequence we conclude that, up to reordering,  $\gcd(f_{j_1},f_{j_2})$ is not a unit in $Q$.  Hence, using the definition of the multiplication in \cref{taylor} on $T$, we have $b_{j_1}\cdot b_{j_2}\in \m T$  and so it follows that in $A$ we have
    \[
    b_{j_1}\cdots b_{j_i}=(b_{j_1}\cdot b_{j_2})\cdots b_{j_i}=0\cdot b_{j_3}\cdots b_{j_i}=0\,.
    \]
    Therefore, the claim holds.
    The argument in \cite[Theorem~2.7]{BEJ2} has already shown
    \[
    \dim \V_R(M)\geqslant n-\ll_{\Lambda}(A)
    \]
    for all $R$-complexes $M$ with $\H(M)$ nonzero and finitely generated. Therefore, combining this fact with the already established claim above, we obtain the desired inequality.
\end{proof}

\begin{remark}
    The same proof works in a more general setting: if for any $i>\height(I)$ and any subset $f_{j_1},\ldots, f_{j_i}$ of $\f$ there exists a pair $f_{j_s},f_{j_t}$ such that 
    \[\frac{f_{j_s}\cdot f_{j_t}}{f_{\{j_s,j_t\}}}\in \m,\] 
    then \Cref{conjecture_realizability} holds.
    One need only replace the Taylor complex $T$ in the proof with the entire Taylor model of $R$ over $Q$ and the rest of the proof remains unchanged. 
\end{remark}

\begin{example}\label{lucho_support_example}
Consider $R = Q/I$ from \Cref{lucho_example}, where $Q = k \llbracket x,y,z,w\rrbracket$ and $I = (x^2, xy, yz, zw, w^2)$. This is a non-Cohen-Macaulay ring with complete intersection defect $2$. \Cref{cid_defect} says that $\dim \V_R(M) \geqslant 2$ for all nonzero complexes $M$ with finitely generated homology. One can easily find $M$ with $\dim \V_R(M) = 3$, such as the cyclic module $M = R/(y,z).$
Indeed, one can compute directly, or apply \cite[Lemma 2.6]{Briggs/Grifo/Pollitz:2022}, to see that the cohomological support variety of $M$ is a 3-dimensional hyperplane. We do not know if there is an $R$-complex with finitely generated homology that has a 2-dimensional cohomological support variety. 
\end{example}

\section{Support varieties for rings defined by monomial ideals}\label{s_support_monomial}

Let us now focus on the support variety $\V_R(R)$ when $R$ is defined by monomials. More precisely, we will let $R$ be a local or positively graded ring with residue field $k$ and a minimal regular presentation $Q/I$, where $I$ is minimally generated by a list of monomials $\f=f_1,\ldots,f_n$ on a regular sequence of $Q$. The Taylor complex on $\f$ is now exact \cite{Taylor}, so we can compute $\V_R(R)$ by determining whether $\widehat{\cC}_{E_{a}}(T)$ is (in)exact for each $a \in \A^n_k$; by \cref{c_support_using_hypersurfaces}, $k$ can be assumed to be algebraically closed. 

We will show that given a fixed $n$, there are only finitely many varieties that can be realized as $\V_R(R)$ where $R$ is defined by monomials; see \cref{support_monomial}. To show that, we will now introduce two useful combinatorial tools to determine $\V_R(R)$ in this context.

\begin{definition}
\label{gcd_graph}
The GCD graph of $\f$ is the simple graph $\Gamma_{\f}$ with vertices $\{ 1, \ldots, n \}$ such that
\[ \text{there is an edge between $i$ and $j$ \quad if and only if \quad} \gcd(f_i, f_j) \text{ is not a unit}.\]
\end{definition}

\begin{remark}\label{gcd_graph_edge}
Note that there is an edge between $i$ and $j$ in $\Gamma_{\f}$ if and only if $b_i b_j\in \m T$ in the Taylor complex $T$ of $\f$; 
    see \Cref{remark product in taylor gcd}.
In these terms, we note that \cref{key lemma} \cref{keylemitem3} says if $n\geqslant 2$ and there is a vertex in $\Gamma_{\f}$ connected to all other vertices, then $\V_R(R)=\mathbb{A}^n$.
\end{remark}

\begin{definition}\label{definition_Taylor_graph}
   The {\bf Taylor graph associated to $\f$} is the directed simple graph defined as follows:

    \begin{itemize}[leftmargin=15pt]
        \item Vertices: one vertex for each subset $J \subseteq [n]$.
        \item Edges: given $J_1, J_2 \subseteq [n]$, there is no edge from $J_1$ to $J_2$ if $|J_1| \neq |J_2| \pm 1$. Given $J \subseteq [n]$ and $i \in [n]$, 
        \begin{enumerate}[label=$\circ$]

            \item Differential edges: If $i \in J$, there is a directed edge from $J$ to $J \smallsetminus \{i\}$ if and only if $f_{J \smallsetminus \{ i \}} = f_J$, or equivalently $f_i \mid f_J$. 
            \item Homotopy edges: If $i \notin J$, there is a directed edge from $J$ to $J \cup \{ i \}$ if and only if $f_i f_J = f_{J \cup \{ i \}}$, or equivalently, $\gcd(f_i, f_j)$ is a unit for all $j \in J$.
        \end{enumerate}
    \end{itemize}
\end{definition}

\begin{remark}\label{remark taylor graph differential}
    The Taylor graph of $\f$ contains all the data necessary to compute the differential in the complex $\widehat{\cC}_{E_{a}}(T)$ for every $a = (a_1, \ldots, a_n) \in \A^n_k$: the edges in the Taylor graph indicate which coefficients in the differential $d$ are nonzero, with differential edges recording unit coefficients in the differential of $T$, and homotopy edges recording units in the homotopies. In the notation of \Cref{taylor-tate}, 
    \[ d(b_J) \quad = \sum_{\substack{i \in J \\ \exists \text{ edge } J \to J \smallsetminus \{ i \} }} \sgn(\{i\}, J) \,\, b_{J \smallsetminus \{ i \}} \quad + \sum_{\substack{i \notin J \\ \exists \text{ edge } J \to J \cup \{ i \} }} \sgn(\{i\}, J) \,\, a_i \, b_{J \cup \{ i \}}.  \]
\end{remark}

\begin{remark}\label{remark taylor vs gcd}
    The GCD graph $\Gamma_{\f}$ of $\f$ contains some of the information encoded in the Taylor graph of $\f$. More precisely, given $J \subseteq [n]$ and $i \notin J$, there is an edge from $J$ to $J \cup \{ i \}$ in the Taylor graph of $\f$ if and only if in the GCD graph of $\f$ there is no edge between $i$ and any $j \in J$.
\end{remark}

\begin{remark}\label{no differentials 2 to 1}
    Whenever the monomials $\f$ form a minimal generating set for $I$, for all $i \neq j$ we have $f_i \nmid f_j$, and thus the Taylor graph of $\f$ has no differential edges of the form $\{ i, j \} \longrightarrow \{ \ell \}$. 
\end{remark}

\begin{remark}\label{remark one neighbor no diffs}
    We claim that if the vertex $i$ of $\Gamma_{\f}$ has degree $1$, then $f_i \nmid f_J$ for all $J \subseteq [n]$ such that $i \notin J$. Indeed, let $j$ be the unique neighbor of $i$ in $\Gamma_{\f}$. Since $\gcd(f_i, f_\ell)$ is a unit for all $\ell \notin \{ i, j \}$, if $f_i \mid f_J$ then we must have $j \in J$ and $f_i \mid f_j$. But this is impossible if the $\f$ are chosen to be a minimal generating set for $I$. We conclude that if the vertex $i$ of $\Gamma_{\f}$ has degree $1$, then the Taylor graph of $\f$ has no differential edges of the form $J \cup \{ i \} \longrightarrow J$. More generally, the same argument shows that if $J$ contains at most one of the neighbors of $i$ in $\Gamma_{\f}$, then $f_i \nmid f_J$ and the Taylor graph of $\f$ does not have the edge $J \longrightarrow J \smallsetminus \{ i \}$.
\end{remark}

\begin{theorem}\label{support_monomial}
There are only finitely many closed subsets of $\mathbb{A}_k^n$ that are realizable as $\V_R(R)$, for any $R$ with regular presentation $Q/I$ where $I$ is generated by $n$ monomials $\f=f_1,\ldots,f_n$ on some regular sequence $\bm{x}$ of $Q$, and the $\f$ correspond to the standard basis of $\mathbb{A}_k^n$. The closed subsets that are realizable depend only on $n$, and are independent of $\bm{x}$ or $Q$.
\end{theorem}

\begin{proof}
It suffices to note that the GCD and LCM lattices of $\f$ completely determine the Taylor graph of $\f$. By \cref{remark taylor graph differential}, $\V_R(R)$ is determined by the Taylor graph, and there are only finitely many directed graphs meeting the specifications of a Taylor graph; the number of allowable graphs is also only dependent on $n$, and does not depend on any properties of the regular sequence $\bm{x}$ nor the ambient regular ring $Q$.
\end{proof}
    
The combinatorial structure of the Taylor graph can be used to quickly determine properties of $\V_R(R)$. In what follows, whenever the defining ideal $I$ of $R$ is minimally generated by $n$ elements and $\V_R(R) = \A^n_k$, we will say that the support of $R$ is full, or simply that $\V_R(R)$ is full.

\begin{lemma}\label{isolated vertex implies full support}
    If the Taylor graph of $\f$ has an isolated vertex, then the support of $R$ is full.
\end{lemma}

\begin{proof}
 Let $J$ be an isolated vertex in the Taylor graph of $\f$, and fix $a \in \A^n_k$ and $C = \widehat{\cC}_{E_{a}}(T)$. Since there are no edges out of $J$ in the Taylor graph of $\f$, by \Cref{remark taylor graph differential} we see that $d(b_J) = 0$, so $b_J$ is a cycle in $C$. Moreover, since there are no edges into $J$, the formula in \Cref{remark taylor graph differential} guarantees that $b_J$ is not a boundary. In fact, the formula gives
\[ \im d \subseteq \mathrm{span}_k \{ b_\ell \mid \ell \subseteq [n], \ell \neq J \}.\] 
Thus $C$ is inexact, and since this was independent of the choice of $a$, we conclude that $\V_R(R) = \A^n_k$.
\end{proof}

If $I$ is a monomial ideal satisfying the conditions of \Cref{key lemma} (\ref{keylemitem3}), then the Taylor graph of $\f$ has an isolated vertex, although the fact that $R$ has full support already follows from \Cref{key lemma}. However, \Cref{isolated vertex implies full support} applies in situations where \Cref{key lemma} (\ref{keylemitem3}) does not:

\begin{example}
Let $Q = k[a,b,c,d,e]$ with $k$ a field. Set
\[f_1 = ab, \quad f_2 = bc, \quad f_3 = cd, \quad f_4 = de,\]
and $I = (f_1, f_2, f_3, f_4)$, the edge ideal of a path of length $4$. One cannot apply \Cref{key lemma} (\ref{keylemitem3}), yet one can check that $\{2, 3\}$ is an isolated vertex of the Taylor graph of $\f$; in fact, this is the only isolated vertex. Here is the Taylor graph of $\f$:

\begin{center}
    \begin{tikzpicture}[
    scale = 0.9,
  node distance=1cm and 1.5cm,
  every node/.style={circle, draw, minimum size=8mm},
  ->,
  baseline=(current bounding box.center)
]

\node (1234) at (0,0) {1234};

\node (123)  at (2,2.5)  {123};
\node (124)  at (2,1)  {124};
\node (134)  at (2,-.5) {134};
\node (234)  at (2,-2) {234};

\node (12) at (4,4)  {12};
\node (13) at (4,2.75)    {13};
\node (14) at (4,1.25)  {14};
\node (23) at (4,-.25)    {23};
\node (24) at (4,-1.75) {24};
\node (34) at (4,-3)   {34};

\node (1) at (6,2.5)  {1};
\node (2) at (6,1)    {2};
\node (3) at (6,-.5)  {3};
\node (4) at (6,-2)   {4};

\node (empty) at (8,0) {$\varnothing$};

\foreach \x in {124,134} {
  \draw (1234) -- (\x);
}

\foreach \x in {1,2,3,4} {
  \draw (empty)-- (\x);
}

\draw (2)--(24);
\draw (4)--(24);
\draw (4)--(14);
\draw (1)--(14);
\draw (3)--(13);
\draw (1)--(13);
\draw (123)--(13);
\draw (12)--(124);
\draw(34)--(134);
\draw (234)--(24);
\end{tikzpicture}

\vspace{0.5em}
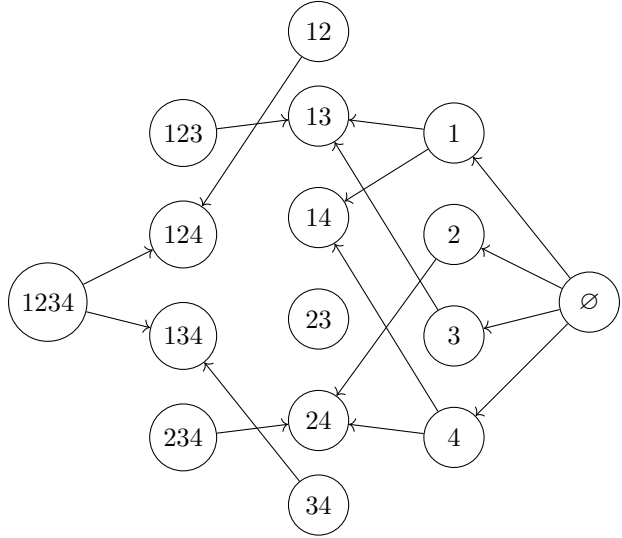
Figure 1: Taylor graph for $(ab,bc,cd,de)$.
\end{center}
\end{example}

The support of $R$ might be full even if the Taylor graph of $\f$ has no isolated vertices.

\begin{example}
\label{spring_example}
Consider $I = (xy,yz,xz)$ in $Q = k \llbracket x, y, z \rrbracket$. The ring $R = Q/I$ does not satisfy the hypothesis of \Cref{key lemma} (\ref{keylemitem3}), and its Taylor graph has no isolated vertices:

\begin{center}
\begin{tikzpicture}[
  node distance=1cm and 1.5cm,
  every node/.style={circle, draw, minimum size=8mm},
  ->,
  baseline=(current bounding box.center)
]

\node (123) at (0,0) {123};

\node (23)  at (2,-1)  {23};
\node (13)  at (2,0)  {13};
\node (12)  at (2,1) {12};

\node (1) at (4,1)  {1};
\node (2) at (4,0)    {2};
\node (3) at (4,-1)  {3};

\node (empty) at (6,0) {$\varnothing$};

\foreach \x in {12,13,23} {
  \draw (123) -- (\x);
}

\foreach \x in {1,2,3} {
  \draw (empty)-- (\x);
}

\end{tikzpicture}

    \vspace{.5em}
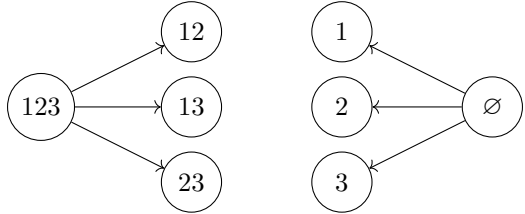
    Figure 2: Taylor graph for $(xy,xz,yz)$
\end{center}

Nevertheless, one can easily see that the support of $R$ is full in many ways: by noting that $b_{\{1,2\}}$ is a cycle but not a boundary in $\widehat{\cC}_{E_{a}}(T)$ for all $a \in \A^3_k$, by doing direct calculations, using Macaulay2, or noting that $R$ is Golod and thus has full support \cite[Theorem 4.1]{Briggs/Grifo/Pollitz:2022}.
\end{example}

\begin{lemma}\label{lemma isolated in degree 2}
    Suppose that there exist $i \neq j$ with the following properties:
    \begin{itemize}
    \item The GCD graph of $\f$ contains an edge between $i$ and $j$.
    \item For all $\ell \in [n] \smallsetminus \{i, j \}$, the GCD graph of $\f$ contains an edge between $\ell$ and exactly one of $i$ or $j$. 
    \end{itemize}
    Then $\V_R(R)$ is full.
\end{lemma}

\begin{proof}
    We will show that $\{ i, j \}$ is an isolated vertex of the Taylor graph on $\f$, which by \Cref{isolated vertex implies full support} implies that $R$ has full support.
    By \Cref{no differentials 2 to 1} there are no edges of the form $\{i , j \} \longrightarrow \{ \ell \}$ for all $\ell \in [4]$.
    
    For $\ell \in [n] \smallsetminus \{ i , j \}$, there is an edge $\{ i, j\} \longrightarrow \{ i, j,\ell \}$ if and only if $\gcd(f_i, f_\ell)$ and $\gcd(f_j, f_\ell)$ are both units. However, we assumed that (exactly) one of these is not a unit, and thus there is no such edge. We have now shown that there are no edges out of $\{ i, j \}$, and thus $b_{\{ i, j \}}$ is a cycle in $\widehat{\cC}_{E_{a}}(T)$ for all $a \in \A^n_k$.

    Our assumption that $i$ and $j$ are connected in the GCD graph means that $\gcd(f_i, f_j)$ is a nonunit, and thus the Taylor graph does not contain the edges $\{ i \} \longrightarrow \{ i, j \}$ and $\{ j \} \longrightarrow \{ i, j \}$. 
    
    It remains to look at potential edges of the form $\{i, j, k\} \longleftarrow \{i, j, \ell \}$ for some $\ell \notin \{ i, j\}$. The Taylor graph contains this edge if and only if $f_\ell$ divides $f_{\{ i, j \}}$. Our assumption on the GCD graph implies that one of $\gcd(f_i, f_\ell)$ of $\gcd(f_j, f_\ell)$ is a unit. But if $\gcd(f_i, f_\ell)$ is a unit and $f_\ell$ divides $f_{\{ i, j \}}$, we must have that $f_\ell$ divides $f_{j}$, which is impossible since we chose $f_1, \ldots, f_n$ to be minimal generators of $I$. We conclude that there is no edge in the Taylor graph of $I$ of the form $\{i, j, k\} \longleftarrow \{i, j, \ell \}$, and this finishes our proof that $\{ i, j \}$ is an isolated vertex in the Taylor graph of $\f$. By \Cref{isolated vertex implies full support}, the support of $R$ is full.
\end{proof}

The last main result of this section is a classification of the varieties realizable as $\V_R(R)$ for some $R$ defined by monomials for small values of $n$. The next lemma will help with that classification.

\begin{lemma}
\label{monomial tor independent}
Let $R$ have regular presentation $Q/I$, with $I$ generated by monomials $\f=f_1,\ldots,f_n$ {on a} regular sequence $\bm{x} = x_1, \ldots, x_d$ of $Q$. 
If $f_1,\ldots,f_m$ are monomials on $x_1,\ldots,x_c$ and $f_{m+1},\ldots, f_n$ are monomials on $x_{c+1},\ldots,x_d$, then under the homeomorphism $\A_k^{n}\cong \A_k^{m}\times_k  \A_k^{n-m}$ we have $\V_R(R)\cong \V_{R_1}(R_1)\times_k \V_{R_2}(R_2)$,
where $R_1$ and $R_2$ have regular presentations $Q/(f_1,\ldots,f_m)$ and $Q/(f_{m+1},\ldots, f_n)$, respectively. 
\end{lemma}

\begin{proof}
It suffices to show that \cref{tor independent gives union of varieties} applies, as $\Tor^Q_i(R_1,R_2)=0$ for $i>0$, in this setting. Indeed, the assumptions on $\f$ provide an isomorphism of dg $Q$-algebras:
    \[
    T(\f)\cong T(f_1,\ldots,f_m)\otimes_Q T(f_{m+1},\ldots,f_n)\,.
    \]
    Moreover, each Taylor complex above only has homology in degree zero since $\f$ is a list of monomials on a regular sequence of $Q$ \cite{Taylor}.
\end{proof}

\begin{theorem}\label{th_nleq4}
Let $R$ be a local or positively graded ring, with a minimal regular presentation $Q/I$. Let $I$ be minimally generated by a list of monomials $\f=f_1,\ldots,f_n$  on a regular sequence of $Q$. 

If $n\leqslant 4$, then the cohomological support variety of $R$ is a coordinate subspace
\[
\V_R(R) = \left\{ a\in \mathbb{A}_k^n\mid a_i=0 \text{ for all } i\in S\right\}
\]
where $S\subseteq\{ 1,\ldots, n\}$ and $c= |S|$ is the largest codimension of an embedded deformation of $R$.
\end{theorem}

\begin{remark}
\Cref{th_nleq4} is false when $n$ is larger than $4$. For instance, the the support variety of the monomial ring $R$ of \Cref{lucho_example} is a union of two hypersurfaces.
\end{remark}

\begin{proof}[Proof of \cref{th_nleq4}]
If $\Gamma_{\f}$ is disconnected, we can write $I$ as a sum of two ideals $I_1$ and $I_2$ in a smaller number of generators that satisfy the hypothesis of \cref{monomial tor independent}, so the support of $R$ can be computed from the supports of $Q/I_1$ and $Q/I_2$. Note that the join of coordinate subspaces is a coordinate subspace, so this reduces the problem to a smaller value of $n$. Thus we may assume that the GCD graph of $\f$ is connected.
   
When $n \leqslant 3$, every connected graph with three vertices or fewer has a vertex connected to all others, and so by part \cref{keylemitem3} of \cref{key lemma} we get $\V_R(R)=\mathbb{A}^n$. Hence we need only to consider the case when $n=4$.
   
Up to renaming the vertices, there are only two connected graphs on four vertices such that no vertex is connected to all others:

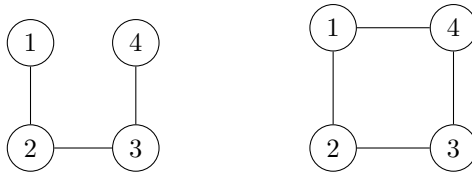
\vspace{-1.2em}
\begin{center}
\begin{tikzpicture}[scale=.7, every node/.style={circle, draw}]
	\node (1) at (0,2) {$1$};
	\node (2) at (0,0) {$2$};
	\node (3) at (2,0) {$3$};
	\node (4) at (2,2) {$4$};

\draw (1) to (2);
\draw (2) to (3);
\draw (3) to (4);
	\end{tikzpicture}
\hspace{5em}
 \begin{tikzpicture}[scale=.8, every node/.style={circle, draw}]
	\node (1) at (0,2) {$1$};
	\node (2) at (0,0) {$2$};
	\node (3) at (2,0) {$3$};
	\node (4) at (2,2) {$4$};

\draw (1) to (2);
\draw (2) to (3);
\draw (3) to (4);
\draw (1) to (4);
\end{tikzpicture}
\end{center}
    \vspace{-0.5em}
In both cases $\{ 2, 3\}$ satisfies the conditions of \Cref{lemma isolated in degree 2}, and thus the support of $R$ is full.
\end{proof}

\begin{theorem}\label{classification of n=5}
Let $R$ be a local or positively graded ring, with a minimal regular presentation $Q/I$, and assume $I$ is minimally generated by five monomials on a regular sequence of $Q$. 
Then the cohomological support variety of $R$ is either a coordinate subspace of $\A^5_k$ or a union of two hyperplanes. More precisely, up to reordering of the generators, 
\[
\V_R(R) = 
\begin{cases}
\mathcal{V}(\chi_1\chi_5) & \text{if }\Gamma_{\f}\text{ is one graph 1 or 2 below and } f_3 \mid f_{24} \\
\left\{ a\in \mathbb{A}_k^n\mid a_i=0 \text{ for all } i\in S\right\} & \text{otherwise}\,, 
\end{cases}
\]
where $S\subseteq\{ 1,\ldots, n\}$ and $c= |S|$ is the largest codimension of an embedded deformation of $R$.

\vspace{0.5em}
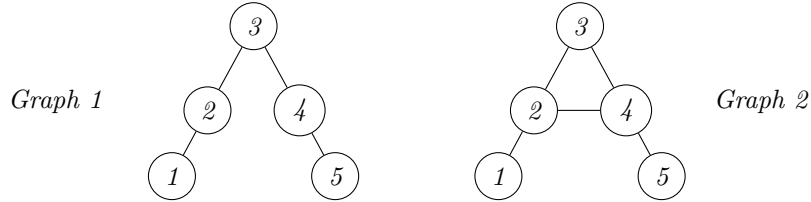
\begin{center}
\begin{minipage}{0.1\textwidth}
    Graph 1
\end{minipage}
\begin{minipage}{0.18\textwidth}

\centering

\vfill

\begin{tikzpicture}[scale=.7, every node/.style={circle, draw,}]
  \node (v1) at (0.45,-1) {1};
  \node (v2) at (1.125,0.25) {2};
  \node (v3) at (2,1.85) {3};
  \node (v4) at (2.875,0.25) {4};
  \node (v5) at (3.55,-1) {5};
  \draw (v1) -- (v2);
  \draw (v2) -- (v3);
  \draw (v4) -- (v5);
  \draw (v3) -- (v4);
\end{tikzpicture}

\vfill
\end{minipage}
\quad\quad\quad 
\begin{minipage}{0.2\textwidth}
\centering

\begin{tikzpicture}[scale=.7, every node/.style={circle, draw}]
  \node (v1) at (0.45,-1) {1};
  \node (v2) at (1.125,0.25) {2};
  \node (v3) at (2,1.85) {3};
  \node (v4) at (2.875,0.25) {4};
  \node (v5) at (3.55,-1) {5};
  \draw (v1) -- (v2);
  \draw (v2) -- (v3);
  \draw (v2) -- (v4);
  \draw (v4) -- (v5);
  \draw (v3) -- (v4);
\end{tikzpicture}
\end{minipage}
\begin{minipage}{0.1\textwidth}
    Graph 2
\end{minipage}
\end{center}
\end{theorem}

\begin{proof}
As we argued in the proof of \Cref{th_nleq4}, the problem reduces to studying ideals whose GCD graph is connected and has no vertex of degree $4$; cf.\@ \Cref{key lemma} (\ref{keylemitem3}). By inspecting a list of all nonisomorphic small graphs such as \cite{ISGCI}, one can check that up to isomorphism, there are 10 connected simple graphs on 5 vertices with no vertex of degree $4$:

\noindent
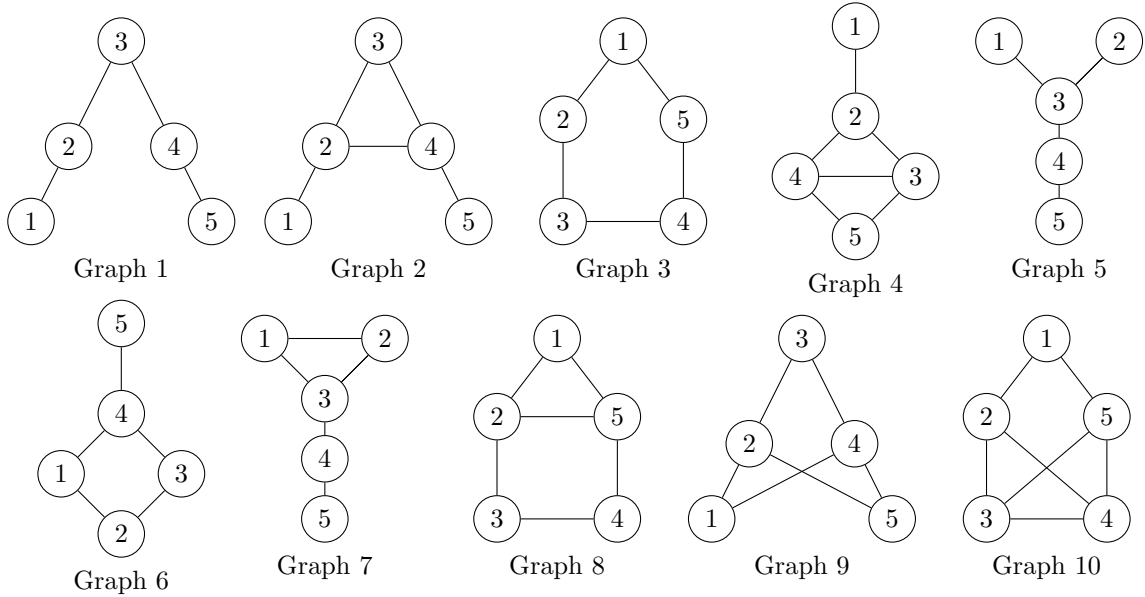
\begin{center}
\begin{minipage}{0.2\textwidth}

\centering

\vfill

\begin{tikzpicture}[scale=.8, every node/.style={circle, draw}]
  \node (v1) at (0.5,-1) {1};
  \node (v2) at (1.125,0.25) {2};
  \node (v3) at (2,2) {3};
  \node (v4) at (2.875,0.25) {4};
  \node (v5) at (3.5,-1) {5};
  \draw (v1) -- (v2);
  \draw (v2) -- (v3);
  \draw (v4) -- (v5);
  \draw (v3) -- (v4);
\end{tikzpicture}
Graph 1

\vfill
\end{minipage}
\begin{minipage}{0.2\textwidth}
\centering

\begin{tikzpicture}[scale=.8, every node/.style={circle, draw}]
  \node (v1) at (0.5,-1) {1};
  \node (v2) at (1.125,0.25) {2};
  \node (v3) at (2,2) {3};
  \node (v4) at (2.875,0.25) {4};
  \node (v5) at (3.5,-1) {5};

  \draw (v1) -- (v2);
  \draw (v2) -- (v3);
  \draw (v2) -- (v4);
  \draw (v4) -- (v5);
  \draw (v3) -- (v4);
\end{tikzpicture}

Graph 2
\end{minipage}
\begin{minipage}{0.18\textwidth}
\centering

\begin{tikzpicture}[scale=.8, every node/.style={circle, draw}]
  \node (v1) at (0,3) {1};
  \node (v2) at (-1,1.7) {2};
  \node (v3) at (-1,0) {3};
  \node (v4) at (1,0) {4};
  \node (v5) at (1,1.7) {5};

  \draw (v1) -- (v2);
  \draw (v2) -- (v3);
  \draw (v3) -- (v4);
  \draw (v4) -- (v5);
  \draw (v5) -- (v1);
\end{tikzpicture}

Graph 3
\end{minipage}
\begin{minipage}{0.18\textwidth}
\centering

\begin{tikzpicture}[scale=.8, every node/.style={circle, draw}]
  \node (v1) at (0,1) {4};
  \node (v2) at (1,0) {5};
  \node (v3) at (2,1) {3};
  \node (v4) at (1,2) {2};
  \node (v5) at (1,3.5) {1};

  \draw (v1) -- (v2);
  \draw (v2) -- (v3);
  \draw (v3) -- (v4);
  \draw (v4) -- (v1);
  \draw (v4) -- (v5);
  \draw (v1) -- (v3);
\end{tikzpicture}

Graph 4
\end{minipage}
\hspace{-5mm}
\begin{minipage}{0.18\textwidth}
\centering

\begin{tikzpicture}[scale=.8, every node/.style={circle, draw}]
  \node (v1) at (1,3) {1};
  \node (v3) at (2,2) {3};
  \node (v2) at (3,3) {2};
  \node (v4) at (2,1) {4};
  \node (v5) at (2,0) {5};

  \draw (v1) -- (v3);
  \draw (v2) -- (v3);
  \draw (v2) -- (v3);
  \draw (v4) -- (v5);
  \draw (v3) -- (v4);
\end{tikzpicture}

Graph 5
\end{minipage}
\begin{minipage}{0.18\textwidth}
\centering

\begin{tikzpicture}[scale=.8, every node/.style={circle, draw}]
  \node (v1) at (0,1) {1};
  \node (v2) at (1,0) {2};
  \node (v3) at (2,1) {3};
  \node (v4) at (1,2) {4};
  \node (v5) at (1,3.5) {5};

  \draw (v1) -- (v2);
  \draw (v2) -- (v3);
  \draw (v3) -- (v4);
  \draw (v4) -- (v1);
  \draw (v4) -- (v5);
\end{tikzpicture}

Graph 6
\end{minipage}
\hspace{-5mm}
\begin{minipage}{0.18\textwidth}
\centering

\begin{tikzpicture}[scale=.8, every node/.style={circle, draw}]
  \node (v1) at (1,3) {1};
  \node (v3) at (2,2) {3};
  \node (v2) at (3,3) {2};
  \node (v4) at (2,1) {4};
  \node (v5) at (2,0) {5};

  \draw (v1) -- (v2);
  \draw (v1) -- (v3);
  \draw (v2) -- (v3);
  \draw (v2) -- (v3);
  \draw (v4) -- (v5);
  \draw (v3) -- (v4);
\end{tikzpicture}

Graph 7
\end{minipage}
\begin{minipage}{0.18\textwidth}
\centering

\begin{tikzpicture}[scale=.8, every node/.style={circle, draw}]
  \node (v1) at (0,3) {1};
  \node (v2) at (-1,1.7) {2};
  \node (v3) at (-1,0) {3};
  \node (v4) at (1,0) {4};
  \node (v5) at (1,1.7) {5};

  \draw (v1) -- (v2);
  \draw (v2) -- (v3);
  \draw (v3) -- (v4);
  \draw (v4) -- (v5);
  \draw (v5) -- (v1);
  \draw (v5) -- (v2);
\end{tikzpicture}

Graph 8
\end{minipage}
\begin{minipage}{0.2\textwidth}
\centering

\begin{tikzpicture}[scale=.8, every node/.style={circle, draw}]
  \node (v1) at (0.5,-1) {1};
  \node (v2) at (1.125,0.25) {2};
  \node (v3) at (2,2) {3};
  \node (v4) at (2.875,0.25) {4};
  \node (v5) at (3.5,-1) {5};

  \draw (v1) -- (v2);
  \draw (v2) -- (v3);
  \draw (v4) -- (v5);
  \draw (v3) -- (v4);
  \draw (v2) -- (v5);
  \draw (v1) -- (v4);
\end{tikzpicture}

Graph 9
\end{minipage}
\begin{minipage}{0.18\textwidth}
\centering

\begin{tikzpicture}[scale=.8, every node/.style={circle, draw}]
  \node (v1) at (0,3) {1};
  \node (v2) at (-1,1.7) {2};
  \node (v3) at (-1,0) {3};
  \node (v4) at (1,0) {4};
  \node (v5) at (1,1.7) {5};

  \draw (v1) -- (v2);
  \draw (v2) -- (v3);
  \draw (v3) -- (v4);
  \draw (v4) -- (v5);
  \draw (v5) -- (v1);
  \draw (v2) -- (v4);
  \draw (v3) -- (v5);
\end{tikzpicture}

Graph 10
\end{minipage}
\end{center}

\vspace{0.5em}
\noindent\underline{\textbf{Graphs 5 to 10}}: For each of these graphs, there is an edge $\{ i, j \}$ that satisfies the hypothesis of \Cref{lemma isolated in degree 2}, and hence $\V_R(R)$ is full in these cases. 

\vspace{0.5em}
\noindent\underline{\textbf{Graph 4}}: We claim that if $\Gamma_{\f}$ is graph 4, then the support of $R$ is full.
To do this, we will explicitly find a cycle that is not a boundary in $C = \widehat{\cC}_{E_{a}}(T)$ for all $a \in \A^5_k$. 

Consider the vertices $\{2, 3\}$ and $\{ 2, 4 \}$ of the Taylor graph of $\f$, and let us find all their neighbors. 
Let $J = \{2, 3\}$ or $J = \{ 2, 4 \}$. Since $|J| = 2$, by \Cref{no differentials 2 to 1} the Taylor graph of $\f$ has no differential edges of the form $J \longrightarrow J \setminus \{ i \}$. Moreover, the structure of $\Gamma_{\f}$ forces $\gcd(f_i, f_{J \smallsetminus\{i\}})$ to be a nonunit for all $i \in [5]$, so the Taylor graph of $\f$ has no homotopy edges of the form $J \longrightarrow J \cup \{ i \}$ or $J \setminus \{ i \} \longrightarrow J$. We conclude that the only possible edges in or out of $J$ are of the form $J \cup \{ i \} \longrightarrow J$, which correspond to $i \notin J$ such that $f_i \mid f_J$. 

By \Cref{remark one neighbor no diffs}, is not possible for $f_1$ or $f_5$ to divide $f_J$. Thus the only possible edges of the Taylor graph involving $\{2, 3\}$ and $\{ 2, 4 \}$ are $\{2, 3, 4\} \longrightarrow \{2, 3\}$ and $\{2, 3, 4\} \longrightarrow \{2, 4\}$; we have these edges, respectively, when $f_4 \mid f_{23}$ and $f_3 \mid f_{24}$. Thus $b_{\{ 2, 3 \}}$ and $b_{\{ 2, 4 \}}$ are cycles in $C$, and the subspace of boundaries contained in the span of $b_{\{ 2, 3 \}}$ and $b_{\{ 2, 4 \}}$ is at most one dimensional. Therefore, $C$ is not exact, and $R$ has full support.

\vspace{0.5em}
\noindent \underline{\textbf{Graph 3}}:
We claim that if $\Gamma_{\f}$ is graph 3, the $5$-cycle, then the support of $R$ is full.

\vspace{0.3em}

For all $i \in [5]$ there is a unique $\{j,\ell\}$ such that $\gcd(f_i, f_{\{j,\ell\}})$ is a unit; in $\Gamma_{\f}$, this corresponds to the unique edge not adjacent to $i$ or its neighbors. Moreover, note that $\{j,\ell\}$ is an edge of $\Gamma_{\f}$, so $\gcd(f_j, f_\ell) \in \m$. A quick analysis of $\Gamma_{\f}$ using \Cref{no differentials 2 to 1,remark one neighbor no diffs} leads us to see that in the Taylor graph of $\f$, the only directed edge in or out of $\{ j, \ell \}$ is $\{ j, \ell \} \longrightarrow \{ i, j, \ell \}$. In particular, this shows that for all $a \in \A^5_k$, given any boundary $w$ in $C = \widehat{\cC}_{E_{a}}(T)$, its unique expression as a linear combination of the $b_J$
\[ w = \sum_{J \subseteq [n]} c_J \cdot b_J\]
has $c_J = 0$ for all $J$ corresponding to an edge of $\Gamma_{\f}$.

Now consider $J \subseteq [5]$ with $|J| = 4$. Since every $i \in [5]$ has two exactly neighbors in $\Gamma_{\f}$, the following hold:
\begin{itemize}
    \item The remaining element $j$ with $[5] \smallsetminus J = \{ j \}$ satisfies $\gcd(f_j,f_J) \in \m$, so there are no edges $J \longrightarrow [5]$.
    \item Any $i \in J$ satisfies $\gcd(f_i, f_{J \smallsetminus \{ i \} }) \in \m$, so the Taylor graph has no edges $J \smallsetminus \{ i \} \longrightarrow J$.
\end{itemize}
Therefore, the only possible edge into any $J$ with $|J| = 4$ is $[5] \longrightarrow J$.

Combining the facts above, it follows that in $C$ there is at most a one dimensional subspace of boundaries contained in
\[
\operatorname{span}_k \{ b_J \mid |J|=4\text{ or } |J|=2 \text{ and }J\text{ is an edge of }\Gamma_{\f}\}.
\]
Using \Cref{remark taylor graph differential},  for each edge $\{ j, \ell \}$ of $\Gamma_{\f}$, if $i$ is the unique vertex not adjacent to $\{ j, \ell\}$ in $\Gamma_{\f}$, then
\[
d(b_{\{ j, \ell \}}) = \sgn(i, \{j, \ell \}) \cdot a_i b_{\{ i, j, \ell \}}.
\]
If $a_i = 0$ for some $i$, then the element $b_{\{ j, \ell \}}$ is a cycle that is not a boundary, and therefore $a \in \V_R(R)$, so we may assume that all $a_i\neq 0$.

We can also compute
\vspace{-0.8em}
\[
d(b_{\{ 1, 2, 4, 5\}}) = \begin{cases}
    b_{\{ 1, 2, 4\}} - b_{\{ 2, 4, 5 \}} & \text{ if } f_1 \mid f_{25} \text{ and } f_5 \mid f_{14} \\
     b_{\{ 1, 2, 4\}} & \text{ if } f_1 \nmid f_{25} \text{ and } f_5 \mid f_{14} \\
    - b_{\{ 2, 4, 5 \}} & \text{ if } f_1 \mid f_{25} \text{ and } f_5 \nmid f_{14} \\
      0 & \text{ if } f_1 \nmid f_{25} \text{ and } f_5 \nmid f_{14},
\end{cases}
\]
and likewise for the other subsets $J\subseteq [5]$ of size $4$. Depending on the cases just listed, one of the following is a cycle:
\[
a_2a_4 b_{\{ 1, 2, 4, 5\}} - a_2 b_{\{ 1, 2 \}} + a_4 b_{\{ 4, 5 \}}, \quad a_4b_{\{ 1, 2, 4, 5\}} - b_{\{ 1, 2 \}}, \quad a_2 b_{\{ 1, 2, 4, 5\}} +  b_{\{ 4, 5 \}}, \quad \text{ or } \quad b_{\{ 1, 2, 4, 5\}}.
\]
Again, there is a similar cycle for the each subset $J\subseteq [5]$ of size $4$. Since all $a_i\neq 0$, these cycles span a five dimensional subspace of
\[
\operatorname{span}_k \{ b_J \mid |J|=4\text{ or } |J|=2 \text{ and }J\text{ is an edge of }\Gamma_{\f}\},
\]
showing that $C$ is not exact for any $a\in \mathbb{A}_k^5$. This finishes the proof that $\V_R(R)$ is full when $\Gamma_{\f}$ is graph 3.

\vspace{0.5em}

\noindent\underline{\textbf{Graphs 1 and 2}}: In these two cases, our argument will involve directly analyzing the differential in the two periodic complexes $C = \widehat{\cC}_{E_{a}}(T)$ for $a \in \A^5_k$. We will write
\[
C = \xymatrix@C=9mm{\cdots \ar[r] & C_{\text{even}} \ar[r]^-{d_{\text{even}}} & C_{\text{odd}} \ar[r]^-{d_{\text{odd}}} &  C_{\text{even}}\ar[r] & \cdots}
\]
where
\[
C_{\text{even}} \colonequals \bigoplus_{|J| \text{ is even}} k \cdot b_J \quad \text{and} \quad C_{\text{odd}} \colonequals \bigoplus_{|J| \text{ is odd}} k \cdot b_J.
\]

\vspace{0.5em}

On the next page, we present the two matrices $d_{\text{even}}$ and $d_{\text{odd}}$ under the canonical choice of bases $b_J$ with $J \subseteq [n]$. 
The matrices for the two GCD graphs are very similar, and so they are simultaneously depicted with the following conventions: 
\begin{itemize}
    \item In column $J$ and row $J \cup \{ i \}$, a nonzero entry is necessarily of the form $\pm a_i$, and its nonvanishing depends only on the choice of GCD graph; any such entry we present is in fact nonzero for graph 1, but circled if \circled{\!it\!} vanishes in the case of graph 2.
    \item In column $J \cup \{ i \}$ and row $J$, any nonzero entry must be of the form $\pm 1$, and its nonvanishing depends on whether $f_i \mid f_J$. We have boxed the \squared{entries} that are nonzero exactly when $f_3 \mid f_{24}$, and circled the remaining such entries, which may vanish or not for both graphs.
\end{itemize}

\newpage

\[ d_{\textrm{even}} = \begin{tblr}{c|[dashed]cccccccccccccccc} 
& {\textrm{\tiny $\varnothing$}} & {\textrm{\tiny 12}} & {\textrm{\tiny 13}} & {\textrm{\tiny 14}} & {\textrm{\tiny 15}} & {\textrm{\tiny 23}} & {\textrm{\tiny 24}} & {\textrm{\tiny 25}} & {\textrm{\tiny 34}} & {\textrm{\tiny 35}} & {\textrm{\tiny 45}} & 
\text{\tiny $\begin{array}{c} 23 \\ 45 \end{array}$} & 
\textrm{\tiny $\begin{array}{c} 13 \\ 45 \end{array}$} 
& \textrm{\tiny $\begin{array}{c} 12 \\ 45 \end{array}$}
& \textrm{\tiny $\begin{array}{c} 12 \\ 35 \end{array}$} 
& \textrm{\tiny $\begin{array}{c} 12 \\ 34 \end{array}$} \\
\hline[dashed]
\textrm{\tiny 1} & a_1 & 0 & 0 & 0 & 0 & 0 & 0 & 0 & 0 & 0 & 0 & 0 & 0 & 0 & 0 & 0 & \\ 
\textrm{\tiny 2} & a_2 & 0 & 0 & 0 & 0 & 0 & 0 & 0 & 0 & 0 & 0 & 0 & 0 & 0 & 0 & 0 & \\ 
\textrm{\tiny 3} & a_3 & 0 & 0 & 0 & 0 & 0 & 0 & 0 & 0 & 0 & 0 & 0 & 0 & 0 & 0 & 0 & \\ 
\textrm{\tiny 4} & a_4 & 0 & 0 & 0 & 0 & 0 & 0 & 0 & 0 & 0 & 0 & 0 & 0 & 0 & 0 & 0 & \\ 
\textrm{\tiny 5} & a_5 & 0 & 0 & 0 & 0 & 0 & 0 & 0 & 0 & 0 & 0 & 0 & 0 & 0 & 0 & 0 & \\
\textrm{\tiny 123} & 0 & 0 & 0 & 0 & 0 & 0 & 0 & 0 & 0 & 0 & 0 & 0 & 0 & 0 & 0 & 0 \\
\textrm{\tiny 124} & 0 & \circled{$a_4$} & 0 & 0 & 0 & 0 & 0 & 0 & 0 & 0 & 0 & 0 & 0 & 0 & 0 & \squared{1} \\
\textrm{\tiny 125} & 0 & a_5 & 0 & 0 & 0 & 0 & 0 & 0 & 0 & 0 & 0 & 0 & 0 & 0 & 0 & 0\\
\textrm{\tiny 134} & 0 & 0 & 0 & 0 & 0 & 0 & 0 & 0 & a_1 & 0 & 0 & 0 & 0 & 0 & 0 & \circled{-1} \\
\textrm{\tiny 135} & 0 & 0 & a_5 & 0 & -a_3 & 0 & 0 & 0 & 0 & a_1 & 0 & 0 & \circled{1} & 0 & \circled{-1} & 0 \\
\textrm{\tiny 145} & 0 & 0 & 0 & 0 & 0 & 0 & 0 & 0 & 0 & 0 & a_1 & 0 & 0 & 0 & 0 & 0 \\ 
\textrm{\tiny 234} & 0 & 0 & 0 & 0 & 0 & 0 & 0 & 0 & 0 & 0 & 0 & 0 & 0 & 0 & 0 & 0 \\
\textrm{\tiny 235} & 0 & 0 & 0 & 0 & 0 & a_5 & 0 & 0 & 0 & 0 & 0 & \circled{1} & 0 & 0 & 0 & 0 \\
\textrm{\tiny 245} & 0 & 0 & 0 & 0 & 0 & 0 & 0 & 0 & 0 & 0 & \circled{$a_2$} & \squared{-1} &  0 & 0 & 0 & 0 \\
\textrm{\tiny 345} & 0 & 0 & 0 & 0 & 0 & 0 & 0 & 0 & 0 & 0 & 0 & 0 & 0 & 0 & 0 & 0 \\
\textrm{\tiny 12345} & 0 & 0 & 0 & 0 & 0 & 0 & 0 & 0 & 0 & 0 & 0 & 0 & 0 & 0 & 0 & 0 
\end{tblr}\]
and
\[ d_{\textrm{odd}} = \begin{tblr}{c|[dashed]cccccccccccccccc} 
& {\textrm{\tiny 1}} & {\textrm{\tiny 2}} & {\textrm{\tiny 3}} & {\textrm{\tiny 4}} & {\textrm{\tiny 5}} & {\textrm{\tiny 123}} & {\textrm{\tiny 124}} & {\textrm{\tiny 125}} & {\textrm{\tiny 134}} & {\textrm{\tiny 135}} & {\textrm{\tiny 145}} & {\textrm{\tiny 234}} & 
\text{\tiny 235} & 
\textrm{\tiny 245} 
& \textrm{\tiny 345}
& \textrm{\tiny 12345} \\
\hline[dashed]
\textrm{\tiny $\varnothing$} & 0 & 0 & 0 & 0 & 0 & 0 & 0 & 0 & 0 & 0 & 0 & 0 & 0 & 0 & 0 & 0 \\
\textrm{\tiny 12} & 0 & 0 & 0 & 0 & 0 & 0 & 0 & 0 & 0 & 0 & 0 & 0 & 0 & 0 & 0 & 0 & \\
\textrm{\tiny 13} & -a_3  & 0 & a_1 & 0 & 0 & \circled{-1} & 0 & 0 & 0 & 0 & 0 & 0 & 0 & 0 & 0 & 0 & \\
\textrm{\tiny 14} & -a_4 & 0 & 0 & a_1 & 0 & 0 & 0 & 0 & 0 & 0 & 0 & 0 & 0 & 0 & 0 & 0 & \\
\textrm{\tiny 15} & -a_5 & 0 & 0 & 0 & a_1 & 0 & 0 & 0 & 0 & 0 & 0 & 0 & 0 & 0 & 0 & 0 & \\ 
\textrm{\tiny 23} & 0 & 0 & 0 & 0 & 0 & 0 & 0 & 0 & 0 & 0 & 0 & 0 & 0 & 0 & 0 & 0 & \\
\textrm{\tiny 24} & 0 & \circled{-$a_4$} & 0 & \circled{$a_2$} & 0 & 0 & 0 & 0 & 0 & 0 & 0 & \squared{-1} & 0 & 0 & 0 & 0 & \\
\textrm{\tiny 25} & 0 & -a_5 & 0 & 0 & a_2 & 0 & 0 & 0 & 0 & 0 & 0 & 0 & 0 & 0 & 0 & 0 & \\
\textrm{\tiny 34} & 0 & 0 & 0 & 0 & 0 & 0 & 0 & 0 & 0 & 0 & 0 & 0 & 0 & 0 & 0 & 0 & \\
\textrm{\tiny 35} & 0 & 0 & -a_5 & 0 & a_3 & 0 & 0 & 0 & 0 & 0 & 0 & 0 & 0 & 0 & \circled{-1} & 0 & \\
\textrm{\tiny 45} & 0 & 0 & 0 & 0 & 0 & 0 & 0 & 0 & 0 & 0 & 0 & 0 & 0 & 0 & 0 & 0 & \\ 
\textrm{\tiny 2345} & 0 & 0 & 0 & 0 & 0 & 0 & 0 & 0 & 0 & 0 & 0 & 0 & 0 & 0 & 0 & 0 \\
\textrm{\tiny 1345} & 0 & 0 & 0 & 0 & 0 & 0 & 0 & 0 & 0 & 0 & 0 & 0 & 0 & 0 & a_1 & \circled{-1} \\
\textrm{\tiny 1245} & 0 & 0 & 0 & 0 & 0 & 0 & 0 & 0 & 0 & 0 & 0 & 0 & 0 & 0 & 0 & \squared{1} \\
\textrm{\tiny 1235} & 0 & 0 & 0 & 0 & 0 & -a_5 & 0 & 0 & 0 & 0 & 0 & 0 & 0 & 0 & 0 & \circled{-1} \\
\textrm{\tiny 1234} & 0 & 0 & 0 & 0 & 0 & 0 & 0 & 0 & 0 & 0 & 0 & 0 & 0 & 0 & 0 & 0
\end{tblr}
\]

Suppose that $a_1 = 0$ or $a_5=0$. Reading from the matrices on the previous page, we see that $b_{\{ 3, 4  \}}$ or $b_{\{2,3\}}$ is a cycle, respectively, as the corresponding column in $d_{\text{even}}$ is zero; furthermore, $b_{\{ 3, 4 \}}$ and $b_{\{2,3\}}$ are never boundaries, as the corresponding rows in $d_{\text{odd}}$ consist only of zeroes. Therefore, $\mathcal{V}(\chi_1 \chi_5) \subseteq\V_R(R)$.

Now consider the case when $f_3 \nmid f_{24}$. This condition forces the column in $d_{\text{odd}}$ corresponding to $234$ to vanish, meaning that $b_{\{ 2, 3, 4 \}}$ is a cycle. But the row corresponding to $234$ in $d_{\text{even}}$ only consists of zeroes, so $b_{\{ 2, 3, 4 \}}$ is not a boundary. This is independent of the choice of $a$, and thus $R$ has full support.
    
Finally, assume that $f_3 \mid f_{24}$. We aim to show that $\V_R(R) = \mathcal{V}(\chi_1\chi_5)$. To this end, assume that $a \in \mathbb{A}_k^5$ with  $a_1\neq 0$ and  $a_5 \neq 0$. By inspection, we see that the rank of $d_{\text{even}}$ is at least 8: columns $\varnothing$, 12, 13, 23, 34, 45, 2345, and 1234 are necessarily linearly independent --- in fact, all but the last one have at least one nonzero entry in a row none of the others do. Similarly, the rank of $d_{\text{odd}}$ is at least $8$: columns 1, 2, 3, 4, 123, 234, 345, and 12345 are necessarily linearly independent. This means that $\ker (d_{\text{even}})$ has rank at most $16-8 = 8$, but since the image of $d_{\text{odd}}$ is contained in $\ker (d_{\text{even}})$ and has rank at least $8$, we conclude that $\ker (d_{\text{even}}) = \im (d_{\text{odd}})$, and both have rank $8$. In particular, $C$ must be exact, and thus $a \notin \V_R(R)$. This finishes the proof of the claim, and the classification of support when $\Gamma_{\f}$ is graph 1 or 2.
\end{proof}

\begin{example}
\label{six-cycle}
   Moving beyond 5 monomial generators, classifying the possibilities for $\V_R(R)$ seems especially difficult considering the number of GCD/Taylor graphs one needs to analyze that have not been previously identified to yield full support. Furthermore, the realizable supports are more complicated. For instance, if $R=k[x_1,\ldots,x_6]/(x_1x_2,x_2x_3,x_3x_4,x_4x_5,x_5x_6,x_6x_1)$, whose corresponding GCD graph is the $6$-cycle, then a Macaulay2 computation yields
   \[
   \V_R(R)=\mathcal{V}(\chi_1\chi_3\chi_5+\chi_2\chi_4\chi_6)\subseteq \spec k[\chi_1,\ldots,\chi_6]\,.
   \]
\end{example}

\begin{remark}
If $R$ is a ring defined by $n$ monomials in the variables of a standard graded polynomial ring $Q$, the polarization of $R$ is a ring $R^\circ$ defined by $n$ squarefree monomials in a (larger) polynomial ring $Q^\circ$, and there are linear forms $\ell_1,\ldots,\ell_m \in Q^\circ$ that form a regular sequence on $R^\circ$, and such that $R^\circ/(\ell_1,\ldots,\ell_m)=R$; see \cite{Froberg:1982}. 
It follows from this construction and \Cref{linear_nzd} that $\V_R(R)=\V_{R^\circ}(R^\circ)$ as subvarieties of $\mathbb{A}^n_k$. For this reason, in classifying what varieties can occur as cohomological support varieties of monomial rings, it suffices to consider squarefree monomials.
\end{remark}

\section*{Acknowledgments}

We thank Mark Walker for helpful discussions regarding the embedded deformation problem and for comments on a previous version of the paper. We also thank Ryan Watson, who found typos on a previous version of this paper.

Finding the proof of \Cref{classification of n=5} involved computing many Taylor graphs of $5$-generated monomial ideals, with $2^5 = 32$ vertices, and analyzing them with the help of the Macaulay2 package \texttt{Visualize} \cite{VisualizeSource}.\looseness -1

Briggs was partly funded by the European Union under the Grant Agreement no.\ 101064551, Hochschild. Grifo was supported by NSF grant DMS-2236983. Pollitz was supported by NSF grant DMS-2302567.

This work started while all three authors were visiting SLMath (formerly known as MSRI), supported by the National Science
Foundation under Grant No.~DMS-1928930 and by the Alfred P.~Sloan Foundation under grant G-2021-16778.

\bibliographystyle{alpha}
\bibliography{references}

\end{document}